\documentclass[12pt]{article}
\title{On dp-minimal fields}
\author{Will Johnson}

\usepackage{amsmath, amssymb, amsthm}    	
\usepackage{fullpage} 	
\usepackage{amscd}
\usepackage{hyperref}
\usepackage[all]{xy}

\DeclareMathOperator*{\forkindep}{\raise0.2ex\hbox{\ooalign{\hidewidth$\vert$\hidewidth\cr\raise-0.9ex\hbox{$\smile$}}}}

\newcommand{\Gal}{\operatorname{Gal}}

\newcommand{\res}{\operatorname{res}}
\newcommand{\Aut}{\operatorname{Aut}}

\newcommand{\acl}{\operatorname{acl}}
\newcommand{\dcl}{\operatorname{dcl}}
\newcommand{\tp}{\operatorname{tp}}

\newcommand{\dpr}{\operatorname{dp-rk}}
\newcommand{\mininf}{-_\infty}

\newtheorem{theorem}{Theorem}[section] 
\newtheorem{lemma}[theorem]{Lemma}
\newtheorem{claim}[theorem]{Claim}
\newtheorem{definition}[theorem]{Definition}
\newtheorem{corollary}[theorem]{Corollary}

\newtheorem{observation}[theorem]{Observation}

\newtheorem{remark}[theorem]{Remark}

\newtheorem{proposition}[theorem]{Proposition}

\newcommand{\Qq}{\mathbb{Q}}

\newcommand{\Zz}{\mathbb{Z}}
\newcommand{\Nn}{\mathbb{N}}
\newcommand{\Cc}{\mathbb{C}}

\newcommand{\Ff}{\mathbb{F}}

\newcommand{\Ll}{\mathbb{L}}

\begin{document}
\maketitle \abstract{We classify dp-minimal pure fields up to
  elementary equivalence.  Most are equivalent to Hahn series fields
  $K((t^\Gamma))$ where $\Gamma$ satisfies some divisibility
  conditions and $K$ is $\mathbb{F}_p^{alg}$ or a local field of
  characteristic zero.  We show that dp-small fields (including
  VC-minimal fields) are algebraically closed or real closed.}
\section{Introduction}
A theory is said to be \emph{VC-minimal} if there is a family
$\mathcal{F}$ of subsets of the home sort such that
\begin{itemize}
\item The family $\mathcal{F}$ is ind-definable without parameters.
\item Every definable subset of the home sort is a boolean combination
  of sets in $\mathcal{F}$.
\item If $X, Y \in \mathcal{F}$, then either $X \subseteq Y$, $Y \subseteq
  X$, or $X \cap Y = \emptyset$.
\end{itemize}

This definition is due to Adler \cite{firstVC}.  Strongly minimal,
C-minimal, o-minimal, and weakly o-minimal theories are all
VC-minimal.

A theory is \emph{not dp-minimal} if there is a model $M$ and formulas
$\phi(x;y)$, $\psi(x;z)$ with $|x| = 1$, and elements $a_{ij}, b_i,
c_j$ such that for all $i, j, i', j'$,
\begin{align*}
  i &= i' \iff M \models \phi(a_{ij},b_{i'}) \\
  j &= j' \iff M \models \psi(a_{ij},c_{j'})
\end{align*}
Otherwise, it is said to be dp-minimal.  Dp-minimality first appeared
in Shelah \cite{ShelahDP} and was isolated as an interesting concept
by Onshuus and Usvyatsov \cite{firstDP}.  It is known that VC-minimal
theories and $p$-minimal theories are dp-minimal---see
\cite{dpExamples}.  Dp-minimality is equivalent to having dp-rank 1;
we will discuss dp-rank in \S \ref{sec:dpr} below.

VC-minimality is not preserved under reducts, but dp-minimality is.

Here are our main results, which essentially classify dp-minimal
fields up to elementary equivalence as pure fields.

\begin{theorem}
  \label{thm-obnoxious}
  Let $(K,v)$ be a henselian defectless valued field, with residue
  field $k$ and value group $\Gamma$ (possibly trivial).  Suppose
  \begin{itemize}
  \item $k \models ACF_p$ or $k$ is elementarily equivalent to a local
    field of characteristic 0.
  \item For every $n$, $|\Gamma/n \Gamma|$ is finite
  \item If $k$ has characteristic $p$, and $\gamma \in [-v(p),v(p)]$
    then $\gamma \in p \cdot \Gamma$.  Here $[-v(p),v(p)]$ denotes
    $\Gamma = (-\infty,\infty)$ if $K$ has characteristic $p$.
  \end{itemize}
  Then $(K,v)$ is dp-minimal as a valued field, and the theory of
  $(K,v)$ is completely determined by the theories of $k$ and $\Gamma$
  (or $k$ and $(\Gamma,v(p))$ in mixed characteristic).
\end{theorem}

The surprising result is that all pure dp-minimal fields arise this way.

\begin{theorem}
  \label{main-result}
  Let $K$ be a sufficiently saturated dp-minimal field.  Then there is
  \emph{some} valuation on $K$ satisfying the conditions of
  Theorem~\ref{thm-obnoxious}.
\end{theorem}
This \emph{almost} says that all dp-minimal fields are elementarily
equivalent to ones of the form $K((t^\Gamma))$ where $K$ is
$\Ff_p^{alg}$ or a characteristic zero local field, and $\Gamma$
satisfies some divisibility conditions.  The one exceptional case is
the mixed characteristic case, which includes fields such as the
spherical completion of $\Zz_p^{un}(p^{1/p^\infty})$.

To obtain a very precise classification of pure dp-minimal fields, we
would need to determine which $(K,v)$ as in
Theorem~\ref{thm-obnoxious} are elementarily equivalent as pure
fields.  (For example, $\Cc((t^\Zz))$ is elementarily equivalent as a
pure field to $\Cc((t^{\Zz \times \Qq}))$, even though $\Zz$ and $\Zz
\times \Qq$ are not elementarily equivalent as ordered groups.)  We
will not do this here---though it seems highly likely that extraneous
factors of $\Qq$ are the only thing that can go wrong.

From the above results, we will obtain the following corollary in \S
\ref{sec:vc}.
\begin{theorem}
  \label{vc-minimal}
  Let $K$ be a VC-minimal field.  Then $K$ is algebraically closed or
  real closed.
\end{theorem}
As we will see, this holds if ``VC-minimal'' is replaced with
Guingona's notion of ``dp-small'' \cite{dpSmall}.

\subsection{Previous work on dp-minimal fields}
Dolich, Goodrick, and Lippel showed that $\Qq_p$ is dp-minimal
\cite{dpExamples}.  John Goodrick \cite{monotonicity} and Pierre Simon
\cite{dpOrdered} proved some results concerning divisible ordered
dp-minimal groups: Goodrick proved an analogue of the monotonicity
theorem for o-minimal structures, and Simon proved that infinite sets
have non-empty interior.  Building off their work, as well as
\cite{weakOmin}, Vince Guingona proved that VC-minimal ordered fields
are real closed \cite{dpSmall}.

Very recently, Walsberg, Jahnke, and Simon have classified dp-minimal
\emph{ordered} fields \cite{JSW}, among other things.  In fact, they
have independently obtained many of the results described below---they
essentially proved our main result Theorem~\ref{main-result}
\emph{modulo} an assumption, which is essentially Theorem~\ref{v-top}
below (see Propositions 7.4 and 8.1 in \cite{JSW}).

\subsection{Outline}
We will focus on Theorem~\ref{main-result}, the truly interesting
result.  Theorem~\ref{thm-obnoxious} is an exercise in quantifier
elimination, though we will include a proof sketch in
\S\ref{sec:obnoxious} below.

In \S \ref{sec:begin} through \S \ref{sec:end}, we will focus on
dp-minimal fields which are not strongly minimal.  The upshot of this
will be the fact that they admit t-henselian V-topologies (essentially
Theorem~\ref{hensel}).  In \S \ref{sec:jkapp} we will apply results of
Jahnke and Koenigsmann \cite{JK} to pick out the desired valuation.
Finally, in \S \ref{sec:ksw-app} we will obtain the divisibility and
defectlessness conditions using results of Kaplan-Scanlon-Wagner
\cite{NIPfields}.

\tableofcontents

\section{Background material}
We review some necessary background material on dp-rank
\S\ref{sec:dpr} and field topologies \S\ref{sec:filters}.
\subsection{Dp-rank}\label{sec:dpr}
If $X$ is a type-definable set and $\kappa$ is a cardinal, a
\emph{randomness pattern of depth $\kappa$ in $X$} is a collection of
formulas $\{\phi_\alpha(x;y_\alpha) : \alpha < \kappa \}$ and elements
$\{ b_{i,j} : i < \kappa, j < \omega\}$ such that for every function
$\eta : \kappa \to \omega$ there is some element $a_\eta$ in $X$ such
that for all $i,j$
\begin{equation*}
  j = \eta(i) \iff \phi_i(a_\eta,b_{ij})
\end{equation*}
The dp-rank of $X$ is defined to be the supremum of the cardinals
$\kappa$ such that there is a randomness pattern of depth $\kappa$ in
$X$.  This definition first appears in \cite{dpRank}.

The following fundamental facts about dp-rank are either easy, or
proven in \cite{dp-add}.
\begin{enumerate}
\item The formula $x = x$ has dp-rank less than $\infty$ if and only
  if the theory is NIP.
\item The formula $x = x$ has dp-rank 1 if and only if the theory is
  dp-minimal.
\item If $X$ is type-definable over $A$, then $\dpr(X)$ is the
  supremum of $\dpr(x/A)$ for $x \in X$.
\item $\dpr(X) > 0$ if and only if $X$ is infinite.
\item For $n < \omega$, $\dpr(a/A) \ge n$ if and only if there are
  sequences $I_1, \ldots, I_n$, which are mutually indiscernible over
  $A$, such that each sequence is not individually $Aa$-indiscernible.
\item Dp-rank is subadditive: $\dpr(ab/A) \le \dpr(a/bA) + \dpr(b/A)$.
\item If $X$ and $Y$ are non-empty type-definable sets, then $\dpr(X
  \times Y) = \dpr(X) + \dpr(Y)$.
\item If $\dpr(a/A) = n$ and $X$ is an $A$-definable set of dp-rank 1,
  then there is $b \in X$ such that $\dpr(ab/A) = n+1$.
\item If $X \twoheadrightarrow Y$ is a definable surjection, then
  $\dpr(Y) \le \dpr(X)$.
\end{enumerate}

Here are some basic uses of dp-rank:
\begin{observation}\label{perfect}
  Let $K$ be a field of finite dp-rank.  Then $K$ is perfect.
\end{observation}
\begin{proof}
  The field $K^p$ of $p$th powers is in definable bijection with $K$,
  so it has the same rank as $K$.  If $K$ is imperfect, then $K$ is a
  definable $K^p$ vector space of dimension greater than 1.  It
  contains a two-dimensional subspace, so $K^p \times K^p$ injects
  definably into $K$.  This shows
  \begin{equation*}
    \dpr(K) \ge 2 \cdot \dpr(K^p) = 2 \cdot \dpr(K)
  \end{equation*}
  So $\dpr(K) = 0$, and $K$ is finite.  Finite fields are perfect.
\end{proof}

\begin{observation}\label{infty-def}
  Let $K$ be dp-minimal field.  Then $K$ eliminates $\exists^\infty$
  (in powers of the home sort).
\end{observation}
\begin{proof}
  It suffices to show that a definable set $X \subset K$ is finite if
  and only if there is some $a \in K$ such that the map $(x,y) \mapsto
  x + a \cdot y$ is injective on $X \times X$.  If $X$ is finite, any
  $a$ outside the finite set
  \begin{equation*}
    \left\{ \frac{x_1 - x_2}{x_3 - x_4} : \vec{x} \in X^4\right\}
  \end{equation*}
  will work.  If $X$ is infinite, then $\dpr(X) \ge 1$, so $\dpr(X
  \times X) = 2$ and $X \times X$ cannot definably inject into $K$.
\end{proof}

This has the following useful corollary:
\begin{corollary}\label{external-infinite}
  Suppose $K$ is dp-minimal.  Then any infinite externally definable
  subset of $K$ contains an infinite internally definable set.
\end{corollary}
\begin{proof}
  Suppose $S \subset K$ is externally definable.  By honest
  definitions (\cite{NIPguide} Remark 3.14), there is some formula
  $\phi(x;y)$ such that for every finite $S_0 \subset S$, there is $b
  \in K$ such that $S_0 \subset \phi(K;b) \subset S$.  By elimination
  of $\exists^\infty$, there is some number $n$ such that $\phi(K;b)$
  is infinite or has size less than $n$.  If we choose $S_0$ to have
  size greater than $n$, then $\phi(K;b)$ will be our desired infinite
  internally-definable set.
\end{proof}

\subsection{Filters and topologies} \label{sec:filters}
Let $K$ be a field, and $\tau$ be a family of subsets of $K$ that is
\emph{filtered}, in the sense that $\forall U, V \in \tau ~ \exists W
\in \tau : W \subset U \cap V$. Consider the following conditions on
$\tau$, lifted straight out of \cite{prestel-ziegler}.  Here,
set-quantifiers range over $\tau$ and element-quantifiers range over
$K$:
\begin{enumerate}
\item \label{f1} $\forall U : \{0\} \subsetneq U$
\item \label{f2} $\forall x \ne 0 ~ \exists U : x \notin U$
\item \label{f3} $\forall U ~ \exists V : V - V \subset U$
\item \label{f4} $\forall U , x ~ \exists V : x \cdot V \subseteq U$
\item \label{f5} $\forall U ~ \exists V : V \cdot V \subseteq U$
\item \label{f6} $\forall U ~ \exists V : (1+V)^{-1} \subseteq (1 + U)$
\item \label{f7} $\forall U ~ \exists V \forall x, y : (x \cdot y \in
  V \implies (x \in U \vee y \in U))$
\end{enumerate}
Then $\tau$ is a neighborhood basis for a Hausdorff non-discrete group
topology on $(K,+)$ if and only if conditions 1-3 hold, and the
topology is uniquely determined in this case.  The topology is a ring
topology if and only if conditions 1-5 hold, a field topology if and
only if 1-6 hold, and a V-topology if and only if 1-7 hold.

If $\tau$ and $\tau'$ are two different filtered families on $K$
satisfying 1-3, then $\tau$ and $\tau'$ will define the same topology
if and only if the following two conditions hold:
\begin{itemize}
\item For all $U \in \tau$ there is $V \in \tau'$ such that $V \subseteq U$.
\item For all $U \in \tau'$ there is $V \in \tau$ such that $V
  \subseteq U$.
\end{itemize}

Now suppose that $K$ is a field with some additional structure, and
the sets in $\tau$ are all definable.  In a saturated elementary
extension $\Cc \succeq K$, the intersection of the sets in $\tau$ is a
set $I_\tau \subset \Cc$ that is type-definable over $K$.  The
conditions above all translate into conditions on $I_\tau$:
\begin{enumerate}
\item \label{t1} $I_\tau \supsetneq \{0\}$
\item \label{t2} $I_\tau$ does not intersect $K^\times$.
\item \label{t3} $I_\tau$ is a subgroup of $(\Cc,+)$
\item \label{t4} $I_\tau$ is closed under multiplication by elements of $K$
\item \label{t5} $I_\tau$ is closed under multiplication
\item \label{t6} $(1 + I_\tau)^{-1} = (1 + I_\tau)$
\item \label{t7} $\Cc \setminus I_\tau$ is closed under multiplication
\end{enumerate}
respectively.  Moreover, $\tau$ and $\tau'$ induce the same topology
if and only if $I_\tau = I_{\tau'}$.  So we have an injective map from
type-definable sets over $K$ satisfying conditions 1-3, to topologies
on $K$.

We will use the following two observations later.
\begin{observation}\label{condition-7}
  Conditions 1-5 and 7 together imply condition 6, and that $I_\tau$
  is the maximal ideal of a valuation ring $\mathcal{O}$ containing $K$.
\end{observation}
\begin{proof}
  Suppose conditions 1-5 and 7 hold.  Let $\mathcal{O}$ be the set of
  $x \in \Cc$ such that $x \cdot I_\tau \subseteq I_\tau$.  This is
  obviously closed under multiplication, and is closed under addition
  and subtraction by condition 3.  So it is a subring of $\Cc$.  It
  contains $I_\tau$ by condition 5, and so $I_\tau$ is an ideal in
  $\mathcal{O}$.  It is a proper ideal by condition 2.  Also,
  $\mathcal{O}$ contains $K$ by condition 4.

  If $x \in \Cc \setminus I_\tau$, then multiplication by $x$
  preserves the complement of $I_\tau$, so division by $x$ preserves
  $I_\tau$.  Thus
  \begin{equation*}
    x \notin I_\tau \implies 1/x \in \mathcal{O}
  \end{equation*}
  Equivalently, $1/x \notin \mathcal{O} \implies x \in I_\tau$.  This
  ensures that $\mathcal{O}$ is a valuation ring with maximal ideal
  $I_\tau$.

  Now it is a general fact that if $\mathfrak{m}$ is the maximal ideal
  of a valuation ring, then $(1+\mathfrak{m})^{-1} = 1 +
  \mathfrak{m}$, so condition 6 holds.
\end{proof}

\begin{observation}\label{funny-comparison}
  Suppose $\tau$ and $\tau'$ satisfy conditions 1-4 and that $a \cdot
  I_\tau \subseteq I_{\tau'}$ for some $a \in \Cc^\times$.  Then
  $I_\tau \subseteq I_{\tau'}$.
\end{observation}
\begin{proof}
  For $U \in \tau'$, we have $a \cdot I_\tau \subseteq I_{\tau'}
  \subseteq U$.  By compactness, there is some $V \in \tau$ such that
  $a \cdot V \subseteq U$.  As $K \preceq \Cc$, there is some $a' \in
  K$ such that $a' \cdot V \subseteq U$.  By condition 4 on $\tau$,
  \begin{equation*}
    I_{\tau} = a' \cdot I_{\tau} \subseteq a' \cdot V \subseteq U
  \end{equation*}
  As $U$ was arbitrary, $I_\tau \subseteq I_{\tau'}$.
\end{proof}

\section{Infinitesimals} \label{sec:begin}
Until \S \ref{sec:six}, let $\Cc$ be a fairly saturated dp-minimal
field \textbf{that is not strongly minimal}.

If $X, Y \subset \Cc$, let $X \mininf Y$ denote
\begin{equation*}
  \{c \in \Cc : \exists^\infty y \in Y : c+y \in X\}
\end{equation*}
This is a subset of $X - Y$.  It is definable if $X$ and $Y$ are, by
Observation~\ref{infty-def}.

\begin{lemma}
  If $X$ and $Y$ are infinite, so is $X \mininf Y$.
\end{lemma}
\begin{proof}
  Suppose $X$ and $Y$ are $A$-definable.  Take $(x,y) \in X \times Y$
  of dp-rank 2 over $A$, and let $c = x - y$.  By subadditivity of
  dp-rank, and dp-minimality,
  \begin{equation*}
    2 = \dpr(x,y/A) = \dpr(y,c/A) \le \dpr(y/c,A) + \dpr(c/A) \le 1 +
    1
  \end{equation*}
  Equality must hold, so $y \notin \acl(Ac)$ and $c \notin \acl(A)$.
  As $y \in Y \cap (X - c)$, the $Ac$-definable set $Y \cap (X - c)$
  is infinite.  Then $c \in X \mininf Y$, so the $A$-definable set $X
  \mininf Y$ is infinite.
\end{proof}

\begin{proposition} \label{basic-infs}
  Let $K \preceq \Cc$ be a small model.  Let
  \begin{equation*}
    \tau = \{X \mininf X : X \subset K \text{ is infinite and $K$-definable}\}
  \end{equation*}
  Then $\tau$ is a filtered family on $K$ and it satisfies conditions
  \ref{f1}, \ref{f2}, \ref{f4} of \S\ref{sec:filters}:
  \begin{enumerate}
  \item $\forall U \in \tau : \{0\} \subsetneq U$
  \item $\forall x \ne 0, ~ \exists U \in \tau : x \notin U$
    \setcounter{enumi}{3}
  \item $\forall U \in \tau,~ \forall x, ~ \exists V \in \tau : x
    \cdot V \subseteq U$
  \end{enumerate}
\end{proposition}
\begin{proof}
  To see $\tau$ is filtered, suppose $X$ and $Y$ are infinite sets.
  As $X \mininf Y$ is non-empty, there is some translate $Y'$ of $Y$
  such that $X \cap Y'$ is infinite.  Then $Y' \mininf Y' = Y \mininf
  Y$ and
  \begin{equation*}
    (X \cap Y') \mininf (X \cap Y') \subseteq (X \mininf X) \cap (Y'
    \mininf Y') = (X \mininf X) \cap (Y \mininf Y)
  \end{equation*}
  so $\tau$ is filtered.

  Condition 1 follows because $0 \in X \mininf X$ for any infinite
  $X$, and $X \mininf X$ is infinite by the lemma.

  Something slightly stronger than condition 4 is true: if $U \in
  \tau$, and $a \in K^\times$, then $a \cdot U \in \tau$.  This
  follows from the identity:
  \begin{equation*}
    (a \cdot X) \mininf (a \cdot X) = a \cdot (X \mininf X)
  \end{equation*}

  In light of this, condition 2 reduces to showing that $X \mininf X
  \ne K$ for some infinite $X$.  By failure of strong minimality and
  Observation\ref{infty-def}, there is a $K$-definable set $D$ which
  is infinite and co-infinite.  Let $D'$ be the complement of $D$.  By
  the Lemma, $D \mininf D'$ is non-empty, so there is some $c$ such
  that $X := D \cap (D' + c)$ is infinite.  Then $X - c \subseteq D'$
  so $(X - c) \cap X = \emptyset$, and $c \notin X \mininf X$.
\end{proof}

We'll denote the corresponding type-definable set by $I_K$, and refer
to elements as \emph{$K$-infinitesimals}.  So $\epsilon \in \Cc$ is
$K$-infinitesimal if and only if $\epsilon \in X \mininf X$ for every
infinite $K$-definable set $X \subset \Cc$.  Equivalently, $X \cap (X
- \epsilon)$ is infinite for every infinite $K$-definable set $X$.

Conditions 1, 2, and 4 of Proposition~\ref{basic-infs} translate into
the following facts: $0 \subsetneq I_K$, $I_K \cap K = \{0\}$, and
$I_K$ is closed under multiplication by $K$.

\subsection{Sums of infinitesimals}
In this section we show that $I_K$ is closed under addition and
subtraction, (condition \ref{t3} of \S \ref{sec:filters}), which
ensures that $\tau$ is a neighborhood basis of a group topology on the
additive group.

\begin{definition}
  Let $K$ be a small model.  A $\Cc$-definable bijection $f : \Cc \to
  \Cc$ is \emph{$K$-slight} if $X \cap f^{-1}(X)$ is infinite for
  every $K$-definable infinite set $X$.
\end{definition}
For example, the translation map $x \mapsto x + \epsilon$ is
$K$-slight if and only if $\epsilon$ is a $K$-infinitesimal.

The main goal here is to show that $K$-slight maps form a group under
composition.

\begin{definition}
  Let $K$ be a small model, and $X \subset \Cc$ be $K$-definable.  Say
  that a $\Cc$-definable bijection $f$ ``\emph{$K$-displaces $X$}'' if
  $X(K) \cap f^{-1}(X)$ is empty.
\end{definition}

\begin{lemma}\label{heir}
  Suppose $K' \succeq K$ and $f'$ and $f$ are $\Cc$-definable
  bijections such that $\tp(f'/K')$ is an heir of $\tp(f/K)$.  (Here,
  we are identifying a bijection with its code.)
  \begin{itemize}
  \item If $f$ is $K$-slight, then $f'$ is $K$-slight.
  \item If $X$ is $K$-displaced by $f$, then $X$ is $K'$-displaced by
    $f'$.
  \end{itemize}
\end{lemma}
\begin{proof}
  First suppose $f$ if $K$-slight.  As $f' \equiv_K f$, the map $f'$
  is also $K$-slight.  If it is not $K'$-slight, there is a
  $K'$-definable infinite set $X$ such that $X \cap (f')^{-1}(X)$ is
  finite.  As $\tp(K'/Kf')$ is finitely satisfiable in $K$, and
  infinity is definable, we can pull the parameters of $X$ into $K$,
  finding a $K$-definable infinite set $X_0$ such that $X_0 \cap
  (f')^{-1}(X_0)$ is finite.  This contradicts $K$-slightness of $f'$.

  Next suppose $X$ is $K$-displaced by $f$.  Then $X$ is $K$-displaced
  by $f'$.  If $X$ is not $K'$-displaced by $f'$, there is some $a \in
  X(K')$ such that $f'(a) \in X$.  As $\tp(a/K f')$ is finitely
  satisfiable in $K$, there is some $a_0 \in X(K)$ such that $f'(a_0)
  \in X$, contradicting the fact that $X$ is $K$-displaced by $f'$.
\end{proof}

\begin{lemma} \label{nip-app}
  No $K$-slight map $K$-displaces an infinite $K$-definable set.
\end{lemma}
\begin{proof}
  Suppose $f_0$ is a $K$-sligiht map which $K$-displaces an infinite
  $K$-definable set $Y$.  Inductively build a sequence of models $K_0 = K
  \preceq K_1 \preceq K_2 \preceq \cdots$ and bijections $f_0, f_1,
  f_2, \ldots$ such that
  \begin{itemize}
  \item $\tp(f_i/K_i)$ is an heir of $\tp(f_0/K)$.
  \item $f_i$ is $K_{i+1}$-definable.
  \end{itemize}
  By Lemma~\ref{heir}, $f_i$ is $K_i$-slight, and $K_i$-displaces $Y$.

  For $w \in \{0,1\}^{< \omega}$, consider the set
  \begin{equation*}
    Y_w = \left \{ y \in Y : \bigwedge_{i < |w|} f_i(y) \in^{w(i)} Y
    \right \}
  \end{equation*}
  where $\in^0$ denotes $\notin$ and $\in^1$ denotes $\in$.

  We will prove by induction on $|w|$ that $Y_w$ is infinite.  If we
  write $f_i$ as $f_{a_i}$, this shows that the formula $f_x(y) \in Y$
  has the independence property, a contradiction.

  For the base case, $Y_{\emptyset}$ is $Y$ which is infinite by
  assumption.

  Now suppose that $Y_w$ is infinite; we will show $Y_{w0}$ and
  $Y_{w1}$ are infinite.  Let $n = |w|$.  Then $Y_w$ is
  $K_n$-definable.  If $a \in Y_w(K_n) \subset Y(K_n)$, then $f_n(a)
  \notin Y$ because $Y$ is $K_n$-displaced by $f_n$.  This shows that
  the infinite set $Y_w(K_n)$ is contained in $Y_{w0}$.

  Also, as $f_n$ is $K_n$-slight and $Y_w$ is infinite and
  $K_n$-definable, $Y_w \cap f_n^{-1}(Y_w)$ is infinite.  This set is
  contained in $Y_w \cap f_n^{-1}(Y) = Y_{w1}$, so $Y_{w1}$ is
  infinite.

  So $Y_w$ being infinite implies $Y_{w0}$ and $Y_{w1}$ are infinite.
  This ensures that all $Y_w$ are infinite, hence non-empty,
  contradicting NIP.
\end{proof}
  
\begin{proposition}\label{slightness}~
  \begin{enumerate}
  \item If $f$ is a $K$-slight bijection and $X$ is $K$-definable,
    then for all but finitely many $x \in K$, we have $x \in X \iff
    f(x) \in X$.
  \item The $K$-slight bijections form a group under composition.
  \end{enumerate}
\end{proposition}
\begin{proof}
  \begin{enumerate}
  \item Let $S \subset K$ be the externally definable set of $x$ such
    that $x \in X$ and $f(x) \notin X$.  We claim that $S$ is finite.
    Otherwise, by Corollary~\ref{external-infinite}, there is some
    infinite $K$-definable set $Y$ such that $Y(K) \subset S$.  Then
    $X \cap Y$ is an infinite $K$-definable set which is $K$-displaced
    by $f$, by choice of $S$.  This contradicts Lemma~\ref{nip-app}.

    So $S$ is finite.  This means that for almost all $x \in K$, we
    have $x \in X \implies f(x) \in X$.  Replacing $X$ with its
    complement, we obtain the reverse implication (with at most
    finitely many exceptions).
  \item Suppose $f$ and $g \circ f$ are $K$-slight.  We will show that
    $g$ is $K$-slight.  Let $X$ be an infinite $K$-definable set.
    Then for almost all $x \in K$, we have
    \begin{equation*}
      f(x) \in X \iff x \in X \iff g(f(x)) \in X 
    \end{equation*}
    So the infinite set $f(X(K))$ is almost entirely contained in $X
    \cap g^{-1}(X)$.  Thus $X \cap g^{-1}(X)$ is infinite, for
    arbitrary infinite $K$-definable sets $X$.
  \end{enumerate}
\end{proof}

\begin{corollary}\label{infinitesimal-sums}
  The set $I_K$ of $K$-infinitesimals is a subgroup of $(\Cc,+)$.  The
  set $I_K$ satisfies conditions \ref{t1}-\ref{t4} of
  \S\ref{sec:filters}.  There is a unique group topology on $(K,+)$
  such that $\{X \mininf X : X \text{ is infinite and
    $K$-definable}\}$ is a neighborhood basis of 0.
\end{corollary}
We will call this topology the \emph{canonical topology} on $K$.  We
may also talk about the canonical topology on $\Cc$, because $\Cc$ is
a model just like $K$.

\section{Germs at 0}
Say that two definable sets $X, Y \subset \Cc$ have the same
\emph{germ at 0} if $0 \notin \overline{X \Delta Y}$.  This is an
equivalence relation.  The main goal of this section is
Theorem~\ref{few-germs}, asserting that there are only a small number
of germs at 0---or equivalently, that there are only a small number of
infinitesimal types over $\Cc$.  This turns out to be the key to
proving a number of basic facts about the canonical topology, as we
will see in \S\ref{subsection-applications}:
\begin{itemize}
\item Definable subsets of $\Cc$ have finite boundary.
\item Products of infinitesimals are infinitesimal.
\item Products of non-infinitesimals are non-infinitesimal.
\item The canonical topology has a definable basis.
\end{itemize}

To prove Theorem~\ref{few-germs}, we would like to mimic Pierre
Simon's argument in the case of ordered dp-minimal structures (Lemma
2.10 in \cite{dpOrdered}).  Matters are complicated by our lack of a
definable neighborhood basis.

In what follows, we'll refer to sets of the form $X \mininf X$ with
$X$ infinite and definable, as ``basic neighborhoods (of 0)''.

Let $\mathcal{U}$ be a 0-definable family of basic neighborhoods (of 0).
\begin{definition}
  Say that $\mathcal{U}$ is \emph{good} if for every finite set $S
  \subset \Cc^\times$, there is some $U \in \mathcal{U}$ such that $U
  \cap S = \emptyset$.
\end{definition}
\begin{definition}
  Say that $\mathcal{U}$ is \emph{mediocre} if for every finite set
  $\{a_1,\ldots,a_n\} \subset \Cc^\times$ of full dp-rank (of dp-rank
  $n$), there is some $U \in \mathcal{U}$ such that $U \cap S =
  \emptyset$.
\end{definition}

A good family would be helpful, but with work, a mediocre family will
suffice.  This is good, because of the following proposition:

\begin{proposition} \label{pseudo-basis}
  There is a mediocre family of basic neighborhoods.
\end{proposition}
\begin{proof}
  Let $\Sigma(x)$ be the partial type over $\Cc$ saying that $x \ne 0$
  and $x$ is a $\Cc$-infinitesimal.

  First suppose that $\Sigma(x)$ is not finitely satisfiable in some
  small model $K$.  Then there is some $\Cc$-definable neighborhood $U
  = U_b$ such that $U_b \cap K = \{0\}$.  Then for all $n$, we have
  \begin{align*}
    \forall a_1, \ldots, a_n \in K^\times : U_b \cap \{a_1, \ldots,
    a_n\} = \emptyset \\
    \forall a_1, \ldots, a_n \in K^\times \exists b \in \Cc : U_b \cap
    \{a_1, \ldots, a_n\} = \emptyset \\
    \forall a_1, \ldots, a_n \in K^\times \exists b \in K : U_b \cap
    \{a_1,\ldots,a_n\} = \emptyset \\
    \forall a_1, \ldots, a_n \in \Cc^\times \exists b \in \Cc : U_b
    \cap \{a_1,\ldots,a_n\} = \emptyset
  \end{align*}
  Consequently the family $\{U_b : b \in \Cc\}$ is a good family of
  basic neighborhoods, hence a mediocre family.

  Therefore, we may assume that $\Sigma(x)$ is finitely satisfiable in
  any small model $K$.  This has the following counterintuitive
  corollary:
  \begin{claim}
    The canonical topology on $K$ is the induced subspace topology
    from the canonical topology on $\Cc$.
  \end{claim}
  \begin{proof}
    The induced subspace topology on $K$ will have as neighborhood
    basis of 0, the sets of the form $N \cap K$ for $N$ a
    $\Cc$-definable basic neighborhood.  This already includes the
    $K$-definable basic neighborhoods on $K$, so it remains to show
    that if $N$ is a $\Cc$-definable basic neighborhood, then there is
    a $K$-definable basic neighborhood $N'$ such that $N' \cap K
    \subset N \cap K$.
    
    By Corollary~\ref{infinitesimal-sums}, applied to the basic
    neighborhoods on $\Cc$, there must be some $\Cc$-definable basic
    neighborhood $U$ such that $U - U \subseteq N$.  We claim $U \cap
    K$ is infinite.  Otherwise, by Hausdorffness we could find a
    smaller $\Cc$-definable neighborhood $V$ such that $V \cap K =
    \{0\}$.  This contradicts the finite satisfiability of $\Sigma(x)$
    in $K$.

    Because $U \cap K$ is infinite, it contains $Q(K)$ for some
    infinite $K$-definable set $Q$, by
    Corollary~\ref{external-infinite}.  Now $Q \mininf Q$ is a
    $K$-definable basic neighborhood, and
    \begin{equation*}
      (Q \mininf Q) \cap K = Q(K) \mininf Q(K) \subseteq Q(K) - Q(K)
      \subseteq U-U \subseteq N
    \end{equation*}
    so that $(Q \mininf Q) \cap K \subseteq N \cap K$.  Then $N' := Q
    \mininf Q$ is our desired $K$-definable basic neighborhood.  This
    proves the claim.
  \end{proof}
  \begin{claim}\label{claim4.5}
    There is a $\emptyset$-definable family of basic neighborhoods
    $U_b$ such that if $K \preceq K'$ is any inclusion of models, and
    $a \in K' \setminus K$, then $(a + U_b) \cap K = \emptyset$ for
    some $b \in K'$.
  \end{claim}
  \begin{proof}
    If not, then by compactness, we would obtain a pair of models $K
    \preceq K'$ and an element $a$ such that every $K'$-definable
    neighborhood of $a$ intersects $K$.  In other words, $a$ is in the
    topological closure $\overline{K}$ of $K$.  Embed $K'$ into $\Cc$.
    Then $K'$ has the induced subspace topology, so $a \in
    \overline{K}$ even within $\Cc$.  Because the topology on $\Cc$ is
    $\Aut(\Cc/K)$-invariant, all the conjugates of $a$ over $K$ are in
    $\overline{K}$, so $\overline{K}$ is big.  But in a Hausdorff
    topology, the closure of a set is bounded in terms of the size of
    the set (because every point in the closure can be written as an
    ultralimit of an ultrafilter on the set, and there are only a
    bounded number of ultrafilters).
  \end{proof}
  Let $U_b$ be the family from Claim~\ref{claim4.5}.  We claim that $U_b$
  is mediocre.  To see this, suppose $a_1, \ldots, a_n$ are elements
  of $\Cc^\times$ with dp-rank $n$ over the empty set.  By properties
  of dp-rank, we can find an element $t \in \Cc$ such that
  $(\vec{a},t)$ has dp-rank $n+1$.

  By subadditivity of dp-rank,
  \begin{align*}
    n+1 & = \dpr(t,t+a_1,\ldots,t+a_n) \\ & \le \dpr(t/t+a_1,\ldots,t+a_n) +
    \dpr(t+a_1,\ldots,t+a_n) \\ & \le 1 + n
  \end{align*}
  so equality holds, and $t \notin \acl(t-a_1,\ldots,t-a_n)$.
  Therefore we can find a small model $K$ such that $t \notin K
  \supseteq \{t+a_1,\ldots,t+a_n\}$.  By the claim there is some $b
  \in \Cc$ such that
  \begin{equation*}
    (t + U_b) \cap \{t+a_1,\ldots,t+a_n\} \subseteq (t + U_b) \cap K =
    \emptyset
  \end{equation*}
  so that $U_b \cap \{a_1, \ldots, a_n\} = \emptyset$.
\end{proof}

\begin{lemma}\label{pseudo-basis-2}
  Let $\mathcal{U}$ be a mediocre family of basic neighborhoods.  Then
  given any small collection $\mathcal{C}$ of infinite definable sets,
  there is some $U \in \mathcal{U}$ such that $C \setminus U$ is
  infinite for every $C \in \mathcal{C}$.
\end{lemma}
\begin{proof}
  Because infinity is definable and $\mathcal{U}$ is a single
  definable family, it suffices by compactness to consider the case
  when $\mathcal{C}$ if a finite collection $\{C_1, \ldots, C_n\}$.
  By definability of infinity, there is some $N$ (depending on
  $\mathcal{C}$) such that $C_i \setminus U$ will be infinite as long
  as it has size at least $N$.

  Let $A$ be a set over which $C_1, \ldots, C_n$ are all defined.  The
  set $\prod_{i = 1}^n C_i^N$ has dp-rank $N \cdot n$, so we can find
  some tuple in it, having dp-rank $N \cdot n$ over $A$, hence over
  $\emptyset$.  By mediocrity, we can find some $U \in \mathcal{U}$
  that $U$ avoids this entire tuple.  By choice of $N$, now each $C_i
  \setminus U$ is infinite.
\end{proof}

\begin{theorem}\label{few-germs}
  There are only a bounded number of germs at 0 among definable
  subsets of $\Cc$.
\end{theorem}
\begin{proof}
  Suppose not.
  \begin{claim}
    There is some sequence $X_1, X_2, \ldots$ of definable subsets of
    $\Cc^\times$, all belonging to a single definable family, such
    that $0 \in \overline{X_i}$ and $0 \notin \overline{X_i \cap X_j}$
    for $i \ne j$.
  \end{claim}
  \begin{proof}
    By Morley-Erdos-Rado, we can produce an indiscernible sequence of
    sets $Y_1, Y_2, Y_3, \ldots \subseteq \Cc$ having pairwise
    distinct germs at 0.  Let $X_i = Y_{2i} \Delta Y_{2i+1}$; then $0
    \in \overline{X_i}$.  By indiscernibility, 0 is in every $Y_i$ or
    in none; either way each $X_i \subseteq \Cc^\times$.

    By NIP, the collection $\{X_i\}$ is $k$-inconsistent for some $k$.
    Replace $X_i$ with $X_{2i} \cap X_{2i+1}$ until $0 \notin
    \overline{X_1 \cap X_2}$.  This process must terminate within
    $\log_2 k$ steps or so.
  \end{proof}
  Fix $X_1, X_2, \ldots$ from the claim.  Let $K_1$ be a small model
  over which the $X_i$ are defined.  Let $\mathcal{U}$ be a mediocre
  famiy from Proposition~\ref{pseudo-basis}.  Inductively build a
  sequence $K_1 \preceq K_2 \preceq \cdots$ and $U_1, U_2, \ldots \in
  \mathcal{U}$ as follows:
  \begin{itemize}
  \item $U_i$ is chosen so that $C \setminus U_i$ is infinite for
    every infinite $K_i$-definable set $C \subseteq \Cc$.  This is
    possible by Lemma~\ref{pseudo-basis-2}.
  \item $K_{i+1}$ is chosen so that $U_i$ is $K_{i+1}$-definable.
  \end{itemize}

  \begin{claim}
    For any $i_0, j_0$, there is some $a$ such that $a \in X_i \iff i
    = i_0$, and $a \in U_j \iff j < j_0$.
  \end{claim}
  \begin{proof}
    By compactness, it suffices to only consider $X_1, \ldots, X_n$
    and $U_1, \ldots, U_n$.  Let
    \begin{equation*}
      D = X_1^c \cap X_2^c \cap \cdots \cap X_{i_0-1}^c \cap X_{i_0}
      \cap X_{i_0 + 1}^c \cap \cdots \cap X_n^c
    \end{equation*}
    where $S^c$ denotes the complement $\Cc \setminus S$ of a set $S$.

    The set $D$ is $K$-definable, and $0 \in \overline{D} \setminus
    D$, by choice of the $X_i$'s.  So the set
    \begin{equation*}
      S = D \cap U_1 \cap \cdots \cap U_{j_0 - 1}
    \end{equation*}
    is infinite, as $U_1 \cap \cdots \cap U_{j_0 - 1}$ is a
    neighborhood of 0.

    As $S$ is $K_{j_0}$ definable, it follows that $S \cap U_{j_0}^c$
    is infinite, by choice of $U_{j_0}$.  As $S \cap U_{j_0}^c$ is
    $K_{j_0 + 1}$-definable, it follows that $S \cap U_{j_0}^c \cap
    U_{j_0 + 1}^c$ is infinite.  Continuing on in this fashion, we
    ultimately see that
    \begin{equation*}
      S \cap U_{j_0}^c \cap \cdots \cap U_n^c
    \end{equation*}
    is infinite.  If $a$ is any element of this set, then $a \in D$,
    so $a \in X_i \iff i = i_0$ (for $1 \le i \le n$), and
    \begin{equation*}
      a \in U_1 \cap \cdots \cap U_{j_0 - 1} \cap U_{j_0}^c \cap
      \cdots \cap U_n^c,
    \end{equation*}
    so $a \in U_j \iff j < j_0$ (for $1 \le j \le n$).

    Finally, using compactness, we can send $n$ to $\infty$.
  \end{proof}
  Given the claim, the sets $\{X_i\}$ and $\{U_i \setminus U_{i+1}\}$
  now directly contradict dp-minimality.
\end{proof}

\begin{corollary}\label{few-infinitesimals}
  There are only a bounded number of infinitesimal types over $\Cc$.
\end{corollary}

\subsection{Applications of bounded germs} \label{subsection-applications}

Using Theorem~\ref{few-germs} and Corollary~\ref{few-infinitesimals},
we can prove a number of key facts about the canonical topology.

We will repeatedly make use of the following basic observation:
\begin{observation}\label{boundary-observation}
  Let $X \subset \Cc$ be $K$-definable, and $a \in K$.  Then the
  following are (clearly) equivalent:
  \begin{enumerate}
  \item There is a $K$-infinitesimal $\epsilon$ such that $(a +
    \epsilon \in X \nLeftrightarrow a \in X)$.
  \item The type $\Sigma(x)$ asserting that $x \in I_K$ and $(a + x
    \in X \nLeftrightarrow a \in X)$ is consistent.
  \item For every $K$-definable basic neighborhood $U$, the set $a +
    U$ intersects both $X$ and $X^c := \Cc \setminus X$.
  \item $a$ is in the topological boundary of $X(K)$ within $K$.
  \end{enumerate}
  Note that the third of these conditions does not depend on $K$, in
  the sense that its truth is unchanged if we replace $K$ with an
  elementary extension $K' \succeq K$.
\end{observation}

First we show that definable sets have finite boundaries.
\begin{proposition}
  \label{finite-boundary}
  If $X \subset K$ is definable, then $\partial X$ is finite, and
  contained in $\acl(\ulcorner X \urcorner)$.
\end{proposition}
\begin{proof}
  By Observation~\ref{boundary-observation}, we may replace $K$ with
  $\Cc$---this only makes $\partial X$ get bigger.

  The set $\partial X$ is type-definable, essentially by (3) of
  Observation~\ref{boundary-observation}.  It is also type-definable
  over $\dcl(\ulcorner X \urcorner)$, by automorphism invariance of
  the topology.  The proposition will therefore follow if $\partial X$
  is small.

  Let $\Cc^*$ be a sufficiently saturated elementary extension of
  $\Cc$.  By the equivalence of conditions 1 and 4 of
  Observation~\ref{boundary-observation},
  \begin{equation} \label{eq:1}
    \partial X(\Cc) = \bigcup_{\epsilon \in I_{\Cc}} \{x \in \Cc : x +
    \epsilon \in X \nLeftrightarrow x \in X\}
  \end{equation}
  Let $D_\epsilon$ denote $\{x \in \Cc : x + \epsilon \in X \nLeftrightarrow
  x \in X\}$.  By the first part of Proposition~\ref{slightness}, each
  $D_\epsilon$ is finite.  Moreover, $D_\epsilon$ depends only on
  $\tp(\epsilon/\Cc)$.  By Corollary~\ref{few-infinitesimals}, it
  follows that the right hand side of (\ref{eq:1}) is small.
\end{proof}

\begin{proposition}\label{ring-topology}
  The set $I_K$ of $K$-infinitesimals is closed under multiplication.
  Consequently, conditions \ref{t1}-\ref{t5} of \S \ref{sec:filters}
  hold and the canonical topology on $K$ is a ring topology.
\end{proposition}
\begin{proof}
  Suppose $\epsilon$ and $e$ are $K$-infinitesimals.
  \begin{claim}
    The map $x \mapsto x \cdot (1 + e)$ is $K$-slight.
  \end{claim}
  \begin{proof}
    Let $X$ be an infinite $K$-definable set; we will show that $X
    \cap (1+e)^{-1}X$ is infinite.  In fact, it contains $X(K)
    \setminus \partial X$, which is infinite by
    Proposition~\ref{finite-boundary}.  To see this, suppose $a \in
    X(K) \setminus \partial X$.  Then $e \cdot a$ is $K$-infinitesimal
    by Proposition~\ref{basic-infs}.  By the equivalence of 1 and 4 in
    Observation~\ref{boundary-observation} and the fact that $a \notin
    \partial X$, it follows that $a + e \cdot a \in X$, so $a \in X
    \cap (1 + e)^{-1}X$.
  \end{proof}
  The $K$-slight maps are closed under composition and inverses, by
  Proposition~\ref{slightness}.  Applying this to the $K$-slight maps
  $x \mapsto (1+e)x$ and $x \mapsto x + \epsilon$, we see that the map
  \begin{equation*}
    x \mapsto (1+e)\left(\frac{x}{1+e} + \epsilon\right) - \epsilon = x + e
    \cdot \epsilon
  \end{equation*}
  is also $K$-slight, so $e \cdot \epsilon$ is a $K$-infinitesimal.
\end{proof}

\begin{lemma}
  \label{G00}
  As a subgroup of the additive group, $I_K$ has no type-definable
  proper subgroups of bounded index.
\end{lemma}
\begin{proof}
  By Proposition 6.1 in \cite{goo}, $I_K^{00}$ exists and is
  type-definable over $K$; we will show $I_K^{00} = I_K$.  Suppose for
  the sake of contradiction that there is $\epsilon \in I_K \setminus
  I_K^{00}$.  Let $K'$ be a model containing $\epsilon$, and let
  $\epsilon'$ realize an heir of $\tp(\epsilon/K)$ to $K'$.  By
  Lemma~\ref{heir}, $\epsilon'$ is $K'$-infinitesimal.

  As $\epsilon$ and $\epsilon'$ have the same (Lascar strong) type
  over $K$, they are in the same coset of $I_K^{00}$.  Then $\epsilon$
  and $\epsilon - \epsilon'$ do \emph{not} have the same type over
  $K$, because the latter is in $I_K^{00}$ but the former is not.
  Choose a $K$-definable set $X$ which contains $\epsilon$ but not
  $\epsilon - \epsilon'$.  As $X$ is $K'$-definable and $\epsilon \in
  K'$, it follows by Observation~\ref{boundary-observation} that
  $\epsilon \in \partial X$.  Then by
  Proposition~\ref{finite-boundary}, $\epsilon \in \acl(\ulcorner X
  \urcorner) \subseteq K$, which is absurd, since $\epsilon$ is a
  non-zero $K$-infinitesimal.
\end{proof}

\begin{theorem}\label{v-top}
  The canonical topology on $K$ is a V-topology.  The set $I_K$
  satisfies conditions \ref{t1}-\ref{t7} of \S\ref{sec:filters}, and
  is the maximal ideal of a valuation ring $\mathcal{O}_K$ on $\Cc$.
\end{theorem}
\begin{proof}
  By Observation~\ref{condition-7} and
  Proposition~\ref{ring-topology}, it suffices to show that the
  complement of $I_K$ is closed under multiplication (condition
  \ref{t7} of \S \ref{sec:filters}).

  First we prove a general fact about dp-minimal groups.
  \begin{claim}\label{4.17}
    Suppose $G$ and $H$ are type-definable subgroups of
    $(K,+)$, such that $G = G^{00}$ and $H = H^{00}$.  Then $G
    \subseteq H$ or $H \subseteq G$.
  \end{claim}
  \begin{proof}
    Otherwise, $G \cap H$ has unbounded index in both $G$ and $H$.  By
    Morley-Erd\"os-Rado we can produce an indiscernible sequence
    $\langle (a_i, b_i) \rangle_{i < \omega + \omega}$ of elements of
    $G \times H$ such that the $a_i$ are in pairwise distinct cosets
    of $G \cap H$, and the $b_i$ are in pairwise distinct cosets of $G
    \cap H$.  The sequences $a_0, a_1, \ldots$ and $b_{\omega},
    b_{\omega + 1}, \ldots$ are mutually indiscernible.  However,
    after naming $c := a_0 + b_\omega$, neither sequence is
    indiscernible.  Indeed, $c - a_i \in H$ if and only if $i = 0$,
    and $b_{\omega + i} - c \in G$ if and only if $i = 0$.  This
    contradicts the characterization of dp-rank 1 in terms of mutually
    indiscernible sequences.
  \end{proof}

  Now, let $R$ be the set of $a \in \Cc$ such that $a \cdot I_K
  \subseteq I_K$.  Observe:
  \begin{enumerate}
  \item $R$ is closed under multiplication, trivially.
  \item $R$ is closed under addition and subtraction, because $I_K$ is
    closed under addition and subtraction.  So $R$ is a ring.
  \item $R$ is a valuation ring in $\Cc$: for any $a \in \Cc^\times$,
    the groups $a \cdot I_K$ and $I_K$ are comparable by
    Claim~\ref{4.17} and Lemma~\ref{G00}.  If $a \cdot I_K \subseteq
    I_K$, then $a \in R$, and if $I_K \subseteq a \cdot I_K$, then
    $1/a \in R$.
  \item $I_K$ is contained in $R$ (by
    Proposition~\ref{ring-topology}), so $I_K$ is an ideal in $R$.
  \item $I_K$ is a proper ideal of $R$, because $1 \notin I_K$.
  \item $K$ is contained in $R$, by Proposition~\ref{basic-infs},
    specifically the fact that $I_K$ is closed under multiplication by
    $K$, which is condition \ref{t4} of \S\ref{sec:filters}.
  \end{enumerate}
  By general facts about valuation rings, the set $J_K = \{x \in \Cc :
  x^2 \in I_K\}$ is also a proper ideal in $R$, and $\Cc \setminus
  I_K$ is closed under multiplication if and only if $I_K = J_K$.

  Because $J_K$ is a proper ideal in a superring of $K$, $J_K$
  satisfies conditions \ref{t2}-\ref{t4} of \S\ref{sec:filters}: it is
  closed under multiplication by $K$, it is closed under addition and
  subtraction, and its intersection with $K$ is $\{0\}$.  As $J_K
  \supseteq I_K$, it also satisfies condition \ref{t1}, the
  non-triviality condition that $J_K \supsetneq \{0\}$.

  So $J_K$ and $I_K$ both satisfy conditions \ref{t1}-\ref{t4} of
  \S\ref{sec:filters}.  Choose nonzero $a \in I_K$.  Then $a \cdot J_K
  \subseteq I_K$ because $I_K$ is an ideal.  By
  Observation~\ref{funny-comparison},
  \begin{equation*}
    a \cdot J_K \subseteq I_K \subseteq J_K
  \end{equation*}
  implies $J_K \subseteq I_K \subseteq J_K$, completing the proof.
\end{proof}

\begin{corollary} \label{definable-topology}
  The canonical topology has a definable basis of opens.  More
  precisely, there is a definable open set $B$ such that sets of the
  form $a \cdot B$ for $a \in \Cc^\times$ form a neighborhood basis of
  0, and consequently the sets of the form $a \cdot B + b$ form a
  basis for the topology.
\end{corollary}
\begin{proof}
  Let $\mathcal{O}$ be the valuation ring whose maximal ideal is
  $I_K$.  For any $x \in \Cc^\times$,
  \begin{equation*}
    x \in \mathcal{O} \iff x^{-1} \notin I_K
  \end{equation*}
  so that $\mathcal{O}$ is $\vee$-definable, over $K$.  By
  compactness, there is some $K$-definable set $B$ such that
  \begin{equation*}
    I_K \subseteq B \subseteq \mathcal{O}
  \end{equation*}
  The inclusion $I_K \subseteq B$ means that $B$ is a neighborhood of
  0.  Replacing $B$ with $B^{int}$ (which is still $K$-definable, by
  Proposition~\ref{finite-boundary}), we may assume that $B$ is open.

  We claim that $\{a \cdot B : a \in K^\times\}$ is a neighborhood
  basis of 0 in the canonical topology on $K$.  Let $U$ be a
  $K$-definable neighborhood of 0.  Take $\epsilon$ a non-zero
  $K$-infinitesimal.  Then 
  \begin{equation*}
    \epsilon \cdot B \subseteq \epsilon \cdot \mathcal{O} \subseteq
    I_K \subseteq U
  \end{equation*}
  As $K \preceq \Cc$, there is some $a \in K^\times$ such that $a
  \cdot B \subseteq U$.  This shows that $\{a \cdot B : a \in
  K^\times\}$ is a neighborhood basis of $0$.  Throwing in translates,
  we get a basis for the topology.
\end{proof}

\begin{definition}
  A \emph{standard ball} is an open definable set $B \subset \Cc$ such
  that $\{a \cdot B : a \in \Cc^\times\}$ is a neighborhood basis of
  0.
\end{definition}

\begin{remark} \label{small-ball}
  Suppose $B$ is a $K$-definable standard ball and $\epsilon$ is a
  $K$-infinitesimal.  Then $\epsilon \cdot B \subseteq I_K$ and hence
  $\epsilon \cdot B$ is contained in any $K$-definable neighborhood of
  0.
\end{remark}
\begin{proof}
  Let $v(-)$ be the valuation on $\Cc$ whose maximal ideal is $I_K$.
  As in the proof of Corollary~\ref{definable-topology}, there is some
  $K$-definable set $B'$ containing the maximal ideal $I_K$ and
  contained in the valuation ring.  In particular, $B'$ is a
  neighborhood of 0 and elements of $B'$ have nonnegative valuation.
  Because $B$ is a standard ball, there is some $a \in \Cc^\times$
  such that $a \cdot B \subseteq B'$.  As $B$ and $B'$ are
  $K$-definable, we can take $a \in K$.  Neither $a$ nor $a^{-1}$ is
  $K$-infinitesimal, so $v(a) = 0$, and we see that $v(x) \ge 0$ for
  any $x \in B$.

  Now if $\epsilon$ is a $K$-infinitesimal, then every element of
  $\epsilon \cdot B$ has positive valuation, so $\epsilon \cdot B
  \subseteq I_K$.
\end{proof}

\section{Henselianity} \label{sec:end}
Let $\mathcal{O}_K$ be the valuation ring whose maximal ideal is
$I_K$.  (This notation is a bit unfortunate, since $\mathcal{O}_K$ is
a valuation ring on $\Cc$, not $K$.  In fact, the valuation is trivial
on $K$.)

In this section, we prove that $\mathcal{O}_K$ is a henselian
valuation ring.

\subsection{Finding interior}

\begin{lemma}\label{no-alg}
  Naming infinitesimals does not algebraize anything, in the following
  sense:
  \begin{enumerate}
  \item Let $\Cc^* \succeq \Cc$ be an elementary extension, and
    $\epsilon \in \Cc^*$ be $\Cc$-infinitesimal.  For any small $S
    \subset \Cc$, we have $\Cc \cap \acl(S \epsilon) = \acl(S)$.
  \item Let $p$ be an infinitesimal type over $\Cc$.  Suppose $S
    \subset \Cc$ is small, $a \in \Cc$, and $\epsilon \models p | S
    a$.  Then $a \in \acl(S) \iff a \in \acl(S \epsilon)$.
  \end{enumerate}
\end{lemma}
\begin{proof}
  \begin{enumerate}
  \item Fix $S$.  For $\epsilon \in I_{\Cc}$, let $X_\epsilon = \acl(S
    \epsilon) \cap \Cc$.  Then $X_\epsilon$ is small and depends only
    on $\tp(\epsilon/\Cc)$.  By Corollary~\ref{few-infinitesimals}, it
    follows that $\bigcup_{\epsilon \in I_\Cc} X_\epsilon$ is small.
    It is also $\Aut(\Cc/S)$-invariant, so it must be contained in
    $\acl(S)$.  In particular, $X_\epsilon \subseteq \acl(S)$ for any
    $\Cc$-infinitesimal $\epsilon$.
  \item Let $\Cc^* \succeq \Cc$ be an elementary extension in which $p$
    is realized by some $\epsilon'$.  Then $\epsilon' \equiv_{aS}
    \epsilon$, so
    \begin{equation*}
      a \in \acl(S \epsilon) \iff a \in \acl(S \epsilon') \iff a \in
      \acl(S)
    \end{equation*}
    where the second equivalence follows by the previous point.
  \end{enumerate}
\end{proof}

If $S$ is a small set, say that an $n$-tuple $(a_1,\ldots,a_n)$ is
\emph{algebraically independent over $S$} if $a_i \notin \acl(S,a_{\ne
  i})$ for each $i$.

\begin{lemma}\label{interior-production}
  Let $S$ be a small set over which some standard ball $B$ is defined.
  Let $(a_1,\ldots,a_n)$ be algebraically independent over $S$.  If
  $\vec{a}$ is in an $S$-definable set $Y \subset \Cc^n$, then
  $\vec{a} \in Y^{int}$ (in the product topology on $\Cc^n$).
\end{lemma}
\begin{proof}
  We proceed by induction on $n$.  The case $n = 1$ is
  Proposition~\ref{finite-boundary}.  Suppose $n > 1$.  Let $K$ be a
  small model containing $S, a_1, \ldots, a_n$.  Let $p$ be some
  global infinitesimal type and let $\epsilon$ realize $p | K$.

  Let $Y$ be the set of $x_1 \in \Cc$ such that $(x_1,a_2,\ldots,a_n)
  \in X$.  Then $Y$ is $S a_2 \cdots a_n$-definable, so $a_1 \notin
  \partial Y$ by Proposition~\ref{finite-boundary}.  Then $Y - a_1$ is
  a $K$-definable neighborhood of 0, so it contains $\epsilon \cdot B$
  by Remark~\ref{small-ball}.  Thus $a_1 + \epsilon \cdot B \subseteq
  Y$.

  Consider
  \begin{equation*}
    Z = \left \{ (x_2,\ldots,x_n) \in \Cc^{n-1} : (a_1 + \epsilon
    \cdot B) \times \{(x_1,\ldots,x_n)\} \subseteq X \right \}
  \end{equation*}
  This set is $S a_1 \epsilon$-definable, and contains
  $(a_2,\ldots,a_n)$.  By Lemma~\ref{no-alg}, $(a_2,\ldots,a_n)$ is
  algebraically independent over $S a_1 \epsilon$, so by induction,
  $(a_2,\ldots,a_n)$ is in the interior of $Z$.  If $U$ is any
  neighborhood of $(a_2,\ldots,a_n)$ in $Z$, then $(a_1 + \epsilon
  \cdot B) \times U$ is a neighborhood of $(a_1,\ldots,a_n)$ in $X$.
\end{proof}

\begin{proposition}\label{weak-open}
  Let $f : \Cc^n \to \Cc^n$ be a finite-to-one definable map.  If $X
  \subset \Cc^n$ is a set with non-empty interior, then $f(X)$ also
  has non-empty interior.  (We are not assuming $X$ is definable.)
\end{proposition}
\begin{proof}
  The topology on $\Cc^n$ has a definable basis, so $X$ must contain a
  definable open.  Shrinking $X$, we may assume $X$ is definable.
  Choose a small model $K$ such that $X$ and $f$ are $K$-definable and
  some standard ball is $K$-definable.  As $X$ has interior, it has
  dp-rank $n$. Choose $\vec{a} \in X$ with $\dpr(\vec{a}/K) = n$ and
  let $\vec{b} = f(\vec{a})$.  The tuple $\vec{b}$ is interalgebraic
  with $\vec{a}$ over $K$, so $\dpr(\vec{b}/K) = n$.  By dp-minimality
  and subadditivity of dp-rank, this implies $\vec{b}$ is
  algebraically independent over $K$ (otherwise, $\vec{b}$ could have
  dp-rank at most $n - 1$).  By Lemma~\ref{interior-production},
  $\vec{b}$ is in the interior of $f(X)$.
\end{proof}

\subsection{Henselianity}

\begin{lemma}\label{vee-extensions}
  Let $F$ be a field with some structure, and $L/F$ be a finite
  extension.  Suppose $\mathcal{O}$ is a $\vee$-definable valuation
  ring on $F$.  Then each extension of $\mathcal{O}$ to $L$ is
  $\vee$-definable (over the same parameters used to define
  $\mathcal{O}$ and interpret $L$).
\end{lemma}
\begin{proof}
  Replacing $L$ with the normal closure of $L$ over $F$, we may assume
  $L/F$ is a normal extension of some degree $n$.

  \begin{claim}
    There is some $d = d(k,n)$ such that the following are equivalent
    for $\{a_1,\ldots,a_k\} \subset L$:
    \begin{itemize}
    \item No extension of $\mathcal{O}$ to $L$ contains
      $\{a_1,\ldots,a_k\}$.
    \item $1 = P(a_1,\ldots,a_k)$ for some polynomial
      $P(X_1,\ldots,X_k) \in \mathfrak{m}[X_1,\ldots,X_k]$ of degree
      less than $d(k,n)$.
    \end{itemize}
  \end{claim}
  \begin{proof}
    Consider the theory $T_n$ whose models consist of degree $n$
    normal field extensions $L/F$ with a predicate picking out a
    valuation ring $\mathcal{O}_L$ on $L$.  On general
    valuation-theoretic grounds, the following are equivalent for
    $\{a_1,\ldots,a_k\} \subset L$
    \begin{itemize}
    \item $\{a_1,\ldots,a_k\} \not \subseteq \sigma(\mathcal{O}_L)$
      for any $\sigma \in \Aut(L/F)$.
    \item No extension of $\mathcal{O}_L \cap F$ to $L$ contains
      $\{a_1,\ldots,a_k\}$.
    \item $1 = P(a_1,\ldots,a_k)$ for some $P(X_1,\ldots,X_k) \in
      \mathfrak{m}[X_1,\ldots,X_k]$.
    \end{itemize}
    The first condition is a definable condition on the $k$-tuple
    $(a_1,\ldots,a_k)$, so by compactness applied to $T_n$, there is a
    bound on the degree in the third condition.
  \end{proof}
  Because $\mathcal{O}$ is $\vee$-definable, $\mathfrak{m}$ is
  type-definable, so the second condition in the claim is
  type-definable.

  Let $\mathcal{O}'$ be some extension of $\mathcal{O}$ to $L$.  We
  can find some finite set $S \subseteq \mathcal{O}'$ such that
  $\mathcal{O}'$ is the unique extension of $\mathcal{O}$ containing
  $S$, because there are only finitely many extensions and they are
  pairwise incomparable.  The claim implies type-definability of the
  set 
  \begin{equation*}
    \{x \in L : \text{no extension of $\mathcal{O}$ to $L$ contains $S
      \cup \{x\}$}\}
  \end{equation*}
  which is the complement of $\mathcal{O}'$ by choice of $S$.
\end{proof}

Recall that $\mathcal{O}_K$ denotes the valuation ring whose maximal
ideal is $I_K$.

\begin{proposition}\label{weak-hensel}
  Let $K$ be a small submodel of $\Cc$.  Let $L/K$ be a finite
  algebraic extension, and $\Ll = L \otimes_K \Cc$.  (So $\Ll$ is a
  saturated elementary extension of $L$.)  Then $\mathcal{O}_K$ has a
  unique extension to $\Ll$.
\end{proposition}
\begin{proof}
  Let $\mathcal{O}_1, \ldots, \mathcal{O}_m$ denote the extensions of
  $\mathcal{O}_K$ to $\Ll$.  By Lemma~\ref{vee-extensions}, these are
  all $\vee$-definable over $K$.  Let $\mathfrak{m}_i$ be the maximal
  ideal of $\mathcal{O}_i$; this is type-definable over $K$.  Let
  $v_i$ be the valuation on $\mathcal{O}_i$.

  Each $v_i$ is a non-trivial valuation, which is trivial when
  restricted to $K$ or even $L$.  It follows easily that each
  $\mathfrak{m}_i$ satisfies conditions \ref{t1}-\ref{t7} of
  \S\ref{sec:filters}.

  Let $I_L$ denote $\bigcap_i \mathfrak{m}_i$.  Then $I_L$ satisfies
  conditions \ref{t2}-\ref{t6} of \S\ref{sec:filters} because the
  $\mathfrak{m}_i$ do, and it satisfies condition \ref{t1}
  (non-triviality) because it contains $I_K$.

  So the $\mathcal{O}_i$ determine V-topologies on $L$, and $I_L$
  determines a field topology on $L$.

  \begin{claim}
    The topology on $I_L$ is the product topology on $L$ (thinking of
    $L$ as a finite-dimensional $K$-vector space.)
  \end{claim}
  \begin{proof}
    Write $L = K(\alpha)$ (possible by Observation~\ref{perfect}).  So
    $\Ll = \Cc(\alpha)$, and $\{1,\alpha,\cdots,\alpha^{n-1}\}$ is a
    basis for $\Ll$ over $\Cc$.

    Let $(F,\mathcal{O})$ be some algebraically closed valued field
    extending $(\Cc,\mathcal{O}_K)$.  Let $\iota_1, \ldots, \iota_n$
    denote the embeddings of $\Ll$ into $F$.  Then
    \begin{align*}
      \{\mathcal{O}_1,\ldots,\mathcal{O}_m\} &= \{
      \iota_i^{-1}(\mathcal{O}) : 1 \le i \le n\} \\
      \{\mathfrak{m}_1,\ldots,\mathfrak{m}_m\} &=
      \{\iota_i^{-1}(\mathfrak{m}) : 1 \le i \le n\}      
    \end{align*}
    where $\mathfrak{m}$ is the maximal ideal of $\mathcal{O}$.  (This
    is merely saying that all the extensions of $\mathcal{O}_K$ to
    $\Ll$ are obtained by embeddings of $\Ll$ into $F$.)

    Because $K \subseteq \mathcal{O}$ (as no element of $K$ is the
    reciprocal of a $K$-infinitesimal), it follows that $K^{alg}
    \subseteq \mathcal{O}$, where $K^{alg}$ is the algebraic closure
    of $K$ inside $F$.  Let $\alpha_1,\ldots,\alpha_n$ be the images
    of $\alpha$ under $\iota_1,\ldots,\iota_n$.  These are pairwise
    distinct because $\Ll/\Cc$ is separable (by
    Observation~\ref{perfect}).  Let $M$ be the Vandermonde matrix
    whose $(i,j)$ entry is $\alpha_i^{j-1}$.  Then $M \in
    GL_n(K^{alg}) \subseteq GL_n(\mathcal{O})$.

    It follows that multiplication by $M$ and $M^{-1}$ preserves
    $\mathcal{O}^n \subseteq F^n$, as well as $\mathfrak{m}^n
    \subseteq F^n$.  Concretely, this means that if
    $(x_0,x_1,\ldots,x_{n-1}) \in F^n$, then the following are
    equivalent:
    \begin{itemize}
    \item Each $x_i \in \mathfrak{m}$
    \item $\sum_{i = 0}^{n-1} x_i \alpha_j^i \in \mathfrak{m}$ for
      each $j$.
    \end{itemize}
    Specializing to the case where $x_0,\ldots,x_{n-1} \in \Cc$, and
    writing $x = \sum_{i = 0}^{n-1} x_i \alpha^i$, the following are
    equivalent:
    \begin{itemize}
    \item The coordinates of $x$ (with respect to the basis
      $\{1,\cdots,\alpha^{n-1}\}$) are in $I_K$
    \item $\iota_j(x) \in \mathfrak{m}$ for each $j \le n$, or
      equivalently, $x \in \mathfrak{m}_i$ for each $i \le m$.
    \end{itemize}
    The latter of these means that $x \in I_L$, so we see that
    \begin{equation*}
      I_L = I_K + I_K \cdot \alpha + \cdots + I_K \cdot \alpha^{n-1}
    \end{equation*}
    from which it is clear that $I_L$ corresponds to the product topology.
  \end{proof}

  Our goal is to prove that $\mathcal{O}_1, \ldots, \mathcal{O}_m$ are
  all equal (i.e., $m = 1$).  Suppose otherwise.

  First suppose that $K$ does not have characteristic 2.  By some fact
  related to Stone approximation (see the proof of (4.1) in
  \cite{prestel-ziegler}), we can find an element $x$ such that $x \in
  1 + \mathfrak{m}_1$ and $x \in -1 + \mathfrak{m}_i$ for $i > 1$.
  Note that $x^2 \in 1 + \mathfrak{m}_i$ for all $i$.

  Thus $x \notin 1 + I_L$, $x \notin -1 + I_L$, and $x^2 \in 1 + I_L$.

  Because $I_L$ defines a field topology, $1 + I_L$ is a subgroup of
  $\Ll^\times$.  It is also topologically open: if $B$ is a standard
  ball and $\epsilon$ is $K$-infinitesimal, then $\epsilon \cdot B
  \subseteq I_K$ (by Remark~\ref{small-ball}), so by the Claim
  \begin{equation*}
  1 + \epsilon \cdot B + \epsilon \cdot B \cdot \alpha + \cdots +
  \epsilon \cdot B \cdot \alpha^{n-1}
  \end{equation*}
  is an open set inside $1 + I_L$.

  The squaring map on $\Ll^\times$ is finite-to-one, so by Proposition
  \ref{weak-open}, $(1 + I_L)^2$ has interior.  Since $(1 + I_L)^2$ is
  a group, it is actually open, hence contains a neighborhood of $1$:
  \begin{equation}
    (1 + I_L)^2 \text{ is a neighborhood of 1} \label{eq2}
  \end{equation}

  Now $x \notin 1 + I_L$ and $-x \notin 1 + I_L$, and $I_L$ is
  type-definable over $K$.  So there is some $K$-definable set $U$
  containing $I_L$, such that $x \notin 1+U$ and $-x \notin 1 + U$.
  By (\ref{eq2}), $(1+U)^2$ is a neighborhood of 0.  It is
  $K$-definable, so it contains $1 + I_L$, hence $x^2$.  Then there is
  $y \in 1 + U$ such that $y^2 = x^2$.  Either $x \in 1 + U$ or $-x
  \in 1 + U$, contradicting the choice of $U$.

  If $K$ has characteristic 2, replace $-1$ and $1$ with $0$ and $1$,
  replace the squaring map with the Artin-Schreier map, and replace $1
  + I_L < \Ll^\times$ with $I_L < \Ll$.
\end{proof}

\begin{lemma}\label{one-topology}
  If $\mathcal{O}$ is a non-trivial definable valuation ring on $\Cc$,
  then $\mathcal{O}$ induces the canonical topology on $\Cc$.  If
  $\mathcal{O}$ is $K$-definable, then $\mathcal{O}_K$ is a coarsening
  of $\mathcal{O}$.
\end{lemma}
\begin{proof}
  It suffices to show that $I_K \subseteq \mathcal{O} \subseteq
  \mathcal{O}_K$, which implies both that $\mathcal{O}_K$ is a
  coarsening of $\mathcal{O}$ and that $\mathcal{O}$ is a standard
  ball as in the proof of Corollary~\ref{definable-topology}.
  
  Let $\mathfrak{m}$ be the maximal ideal of $\mathcal{O}$.  It is
  infinite, since $\mathcal{O}$ is non-trivial.  By
  Proposition~\ref{finite-boundary}, $\mathfrak{m}$ has interior.  As
  $\mathfrak{m}$ is a subgroup of the additive group, $\mathfrak{m}$
  is open.  Then $0$ is in the interior of $\mathfrak{m}$, meaning
  that $I_K \subseteq \mathfrak{m}$.  This directly implies
  $\mathcal{O} \subseteq \mathcal{O}_K$.
\end{proof}

\begin{remark}
  \label{joining-vee-defables}
  Suppose $F$ is a field with some structure, and $\mathcal{O}_1$ and
  $\mathcal{O}_2$ are incomparable $\vee$-definable valuation rings on
  $F$.  Then the join $\mathcal{O}_1 \mathcal{O}_2$ is
  \emph{definable}.
\end{remark}
\begin{proof}
  The join can be written as either $\{x \cdot y : x \in
  \mathcal{O}_1, y \in \mathcal{O}_2\}$ (which is $\vee$-definable) or
  as $\{x \cdot y : x \in \mathfrak{m}_1, y \in \mathfrak{m}_2\}$,
  which is type-definable.\footnote{Here, we are using the fact that
    if $\mathcal{O}$ is a valuation ring with maximal ideal
    $\mathfrak{m}$, and $S$ is any set, then $S \cdot \mathcal{O}$ and
    $S \cdot \mathfrak{m}$ are closed under addition, and are equal to
    each other unless $S$ has an element of minimum valuation.
    Incomparability of $\mathcal{O}_1$ and $\mathcal{O}_2$ ensures
    that e.g. $v_2(\mathcal{O}_2)$ has no minimum.}
\end{proof}

\begin{lemma} \label{no-independent}
  Let $\Ll/\Cc$ be a finite algebraic extension.  Any two non-trivial
  definable valuation rings on $\Ll$ are \emph{not} independent, i.e.,
  they induce the same topology.
\end{lemma}
\begin{proof}
  Let $w_1, w_2$ be two definable valuations on $\Ll$, and let $v_1$
  and $v_2$ be their restrictions to $\Cc$.  Let $\Gamma_i$ be the
  value group of $w_i$.  Let $K$ be a small model over which
  everything is defined (including the extension $\Ll/\Cc$).  Let
  $v_K$ be the non-definable valuation on $\Cc$ coming from
  $\mathcal{O}_K$ and $I_K$.  By Lemma~\ref{one-topology}, $v_K$ is a
  coarsening of $v_1$ and $v_2$.  So there are convex subgroups
  $\Delta_i < \Gamma_i$ such that $v_K$ is equivalent to the
  coarsening of $v_i$ by $\Delta_i$.  Let $w'_i$ be the coarsening of
  $w_i$ by $\Delta_i$.  Then $w'_1$ and $w'_2$ are valuations on $\Ll$
  extending $v_K$.  By Proposition~\ref{weak-hensel}, $w'_1$ and
  $w'_2$ are equivalent (because $v_K$ has an essentially unique
  extension).  It follows that $w_1$ and $w_2$ have a common
  coarsening---the unique extension of $v_K$ to $\Ll$.  This common
  coarsening is non-trivial, because $v_K$ is non-trivial.
  Non-trivial coarsenings induce the same topology, so $w_1, w'_1,$
  and $w_2$ all induce the same topology.  Therefore $w_1$ and $w_2$
  are not independent.
\end{proof}

\begin{proposition} \label{comparability}
  Let $\Ll$ be a finite extension of $\Cc$.  Any two definable
  valuation rings on $\Ll$ are comparable.
\end{proposition}
\begin{proof}
  Suppose $\mathcal{O}_1$ and $\mathcal{O}_2$ are incomparable.  Let
  $\mathcal{O} = \mathcal{O}_1 \cdot \mathcal{O}_2$ be their join,
  which is definable by Remark~\ref{joining-vee-defables}.  Let $w$ be
  the valuation corresponding to $\mathcal{O}$, and let $v$ be its
  restriction to $\Cc$.

  The residue field $\Ll' := \Ll w$ is a finite extension of $\Cc' :=
  \Cc v$.  Moreover, $\Ll'$ has two independent definable valuations,
  induced by $\mathcal{O}_1$ and $\mathcal{O}_2$.  This ensures that
  $\Ll'$ is infinite and unstable, so $\Cc'$ is also infinite and
  unstable.  But $\Cc'$ has dp-rank at most 1, so $\Cc'$ is a
  dp-minimal unstable field.  It is also as saturated as $\Cc$, so all
  our results so far apply to $\Cc'$.  By Lemma~\ref{no-independent},
  $\Ll'$ cannot have two independent definable valuation rings, a
  contradiction.
\end{proof}

\begin{corollary}\label{auto-hensel}
  Any definable valuation ring $\mathcal{O}$ on $\Cc$ is henselian.
\end{corollary}
\begin{proof}
  Otherwise, $\mathcal{O}$ would have two incomparable extensions to
  some finite Galois extension of $\Cc$.
\end{proof}
Corollary~\ref{auto-hensel} was obtained independently by Jahnke,
Simon, and Walsberg (Proposition 4.5 in \cite{JSW}).

\begin{theorem}
  \label{hensel}
  The valuation ring $\mathcal{O}_K$ (whose maximal ideal is the set
  of $K$-infinitesimals) \emph{is henselian}.
\end{theorem}
\begin{proof}
  Suppose not.  Then $\mathcal{O}_K$ has multiple extensions to some
  finite algebraic extension $\Ll/\Cc$.  Let $\mathcal{O}_1$ and
  $\mathcal{O}_2$ be two such extensions.  Let $K' \succeq K$ be a
  larger model over which the field extension $\Ll/\Cc$ is defined.
  As $I_{K'} \subseteq I_K$, we see that $\mathcal{O}_{K'}$ is a
  coarsening of $\mathcal{O}_K$.  Also, $\mathcal{O}_{K'}$ has a
  unique extension to $\Ll$ by Proposition~\ref{weak-hensel}.  As in
  the proof of Lemma~\ref{no-independent}, this ensures that
  $\mathcal{O}_1$ and $\mathcal{O}_2$ are not independent.  Their join
  $\mathcal{O}_1 \cdot \mathcal{O}_2$ is definable by
  Lemma~\ref{vee-extensions} and Remark~\ref{joining-vee-defables}.
  It is also non-trivial because $\mathcal{O}_1$ and $\mathcal{O}_2$
  aren't independent.

  So there is some definable non-trivial valuation ring on $\Cc$.  The
  property of being a valuation ring is expressed by finitely many
  sentences, and $K \preceq \Cc$, so there is a $K$-definable
  non-trivial valuation ring $\mathcal{O}$.  This ring is henselian by
  Corollary~\ref{auto-hensel}, and $\mathcal{O}_K$ is a coarsening, by
  Lemma~\ref{one-topology}.  Coarsenings of henselian valuations are
  henselian.
\end{proof}

\subsection{Summary of results so far}
In what follows, we will need only the following facts from
\S\ref{sec:begin}-\S\ref{sec:end}:
\begin{theorem} \label{summary}
  Let $K$ be a dp-minimal field.
  \begin{enumerate}
  \item $K$ is perfect.
  \item If $K$ is sufficiently saturated and not algebraically closed,
    then $K$ admits a non-trivial Henselian valuation (not necessarily
    definable).
  \item Any definable valuation on $K$ is henselian.  Any two
    definable valuations on $K$ (or any finite extension of $K$) are
    comparable.
  \item For any $n$, the cokernel of the $n$th power map $K^\times \to
    K^\times$ is finite.
  \end{enumerate}
\end{theorem}
\begin{proof}~
  \begin{enumerate}
  \item Observation~\ref{perfect}.
  \item If $K$ is strongly minimal, then $K$ is algebraically closed
    by a well-known theorem of Macintyre.  Otherwise, this is
    Theorem~\ref{hensel}.
  \item If $K$ isn't strongly minimal, this is
    Proposition~\ref{comparability} and Corollary~\ref{auto-hensel}.
    Otherwise, $K$ is NSOP, so has only the trivial valuation.
  \item If $K$ is strongly minimal, then $K$ is algebraically closed
    (Macintyre), so the cokernels are always trivial.  If
    $K^\times/(K^\times)^n$ is infinite, we can find some elementary
    extension $M \succeq K$ such that $M^\times/(M^\times)^n$ is
    greater in cardinality than the total number of infinitesimal
    types over $\Cc$, by Corollary~\ref{few-infinitesimals}.  By
    Lemma~\ref{heir}, heirs of infinitesimal types are infinitesimal
    types, so $\Cc$ has at least as many infinitesimal types as $M$,
    and therefore the cardinality of $M^\times/(M^\times)^n$ exceeds
    the number of infinitesimal types over $M$.  Now for any $a \in
    M^\times$, and any $M$-infinitesimal $\epsilon$, the element $a
    \cdot \epsilon^n$ is an $M$-infinitesimal in the same coset as
    $a$.  So there are $M$-infinitesimals in every coset of
    $(M^\times)^n$, contradicting the choice of $M$.
  \end{enumerate}
\end{proof}

\section{The proof of Theorem~\ref{main-result}} \label{sec:six}

\subsection{Review of Jahnke-Koenigsmann}
First we review some facts and definitions from \cite{JK}.

Following \cite{JK}, if $K$ is any field, let $K(p)$ denote the
compositum of all $p$-nilpotent Galois extensions of $K$.  Let's say
that $K$ is ``$p$-closed'' if $K = K(p)$.  The map $K \mapsto K(p)$ is
a closure operation on the subfields of $K^{alg}$.  By an analogue of
the Artin-Schreier theorem, if $[K(p) : K]$ is finite, then $K$ is
$p$-closed or orderable.  Say that a field $K$ is ``$p$-jammed'' if no
finite extension is $p$-closed.

\begin{remark} \label{p-jam}
  If $K$ is not real closed or separably closed, then $K$ has a finite
  extension which is $p$-jammed for some prime $p$.
\end{remark}
\begin{proof}
  Replace $K$ with $K(\sqrt{-1})$ in characteristic 0.  Take some
  non-trivial finite Galois extension $L/K$.  Take $p$ dividing
  $|\Gal(L/K)|$.  By Sylow theory there is some intermediate field $K
  < F < L$ such that $L/F$ is a $p$-nilpotent Galois extension.  Then
  $F(p) \ne F$, so $[F(p) : F] = \infty$ because $F$ isn't orderable.
  No finite extension $F'$ of $F$ will contain $F'(p) \supseteq F(p)$,
  so $F$ is $p$-jammed.
\end{proof}

Following \cite{JK}, a valuation on a field $K$ is
\emph{$p$-henselian} if it has a unique extension to $K(p)$.  On any
field $K$, there is a \emph{canonical $p$-henselian valuation}
$v^p_K$.  If the residue field $K v^p_K$ is not $p$-closed, then
$v^p_K$ is the finest $p$-henselian valuation on $K$.  By Main Theorem
3.1 of \cite{JK}, $v^p_K$ is 0-definable, provided that $X^{p^2}-1$
splits in $K$, and $K(p) \ne K$.

\subsection{Applying Jahnke-Koenigsmann} \label{sec:jkapp}

\begin{theorem} \label{jk-app}
  Let $K$ be a sufficiently saturated dp-minimal field.  Let
  $\mathcal{O}_\infty$ be the intersection of all the definable
  valuation rings on $K$.  (So $\mathcal{O}_\infty = K$ if $K$ admits
  no definable non-trivial valuations.)
  \begin{enumerate}
  \item $\mathcal{O}_\infty$ is a henselian valuation ring on $K$
  \item $\mathcal{O}_\infty$ is type-definable, without parameters.
  \item \label{tres} The residue field of $\mathcal{O}_\infty$ is
    finite, real-closed, or algebraically closed.  If it is finite,
    then $\mathcal{O}_\infty$ is definable.
  \end{enumerate}
\end{theorem}
\begin{proof}
  \begin{enumerate}
  \item By Theorem~\ref{summary}(3), the class of definable valuation
    rings on $K$ is totally ordered.  An intersection of a chain of
    valuation rings is a valuation ring.  An intersection of a chain
    of henselian valuation rings is henselian.
  \item We need to show that $\mathcal{O}_\infty$ is a small
    intersection.  Suppose $\mathcal{O}$ is a definable valuation ring
    on $K$, defined by a formula $\phi(K;b)$.  Let $\psi(x)$ be the
    formula asserting that $\phi(K;x)$ is a valuation ring.  Then
    $\bigcap_{b \in \psi(K)} \phi(K;b)$ is a $\emptyset$-definable
    valuation ring contained in $\mathcal{O}$.  Thus every definable
    valuation ring on $K$ contains a 0-definable valuation ring.
    Therefore $\mathcal{O}_\infty$ is the intersection of the
    0-definable valuation rings on $K$.  It is therefore
    type-definable over $\emptyset$.
  \item Let $v_\infty$ denote the valuation on $K$ corresponding to
    $\mathcal{O}_\infty$.  If $v$ is a valuation on $K$, let $Kv$
    denote the residue field of $v$.  If $v$ is henselian and $L/K$ is
    a finite extension, write the residue field of $v$'s unique
    extension to $L$ as $Lv$, by abuse of notation.

    We'll refer to a map $K \cup \{\infty\} \to k \cup \{\infty\}$
    coming from a residue map, as a ``place'' from $K$ to $k$.

    Write $v \le v'$ to indicate that $v$ is a coarsening of $v'$.  In
    this case, there are places $Kv \to Kv'$ and even $Lv \to Lv'$.
    If $v$ is definable, then $v \le v_\infty$ by choice of
    $\mathcal{O}_\infty$.

    If $L/K$ is a finite extension, the definable valuations on $L$
    are totally ordered, and in strict order-preserving bijection with
    the definable valuations on $K$, essentially by
    Theorem~\ref{summary}.

    \begin{remark} \label{confusion}
      If $v$ is definable and $Lv \to k$ is a definable place, then
      $Lv \to k$ is equivalent to $Lv \to Lv'$ for some definable $v'
      \ge v$ on $K$.
    \end{remark}
    (To see this, compose the places $L \to Lv \to k$ to get a
    definable valuation on $L$; let $v'$ be its restriction to $K$.)

    With these preliminaries out of the way, we first show that $Kv_\infty$ is
    perfect.  Suppose $Kv_\infty$ has characteristic $p$.  Then $Kv_1$
    has characteristic $p$ for some definable valuation $v_1$.
    (Otherwise, $K$ has characteristic 0 and $1/p \in \mathcal{O}$ for
    every definable valuation ring on $K$.  So $1/p \in
    \mathcal{O}_\infty$, and $Kv_\infty$ has characteristic 0.)  The
    field $Kv_1$ is finite or dp-minimal, so by Theorem~\ref{summary},
    it is perfect.  The place $Kv_1 \to Kv_\infty$ now ensures that
    $Kv_\infty$ is perfect as well.

    Now we turn to proving part \ref{tres} of the theorem.  First
    suppose that $Kv$ is finite for some definable $v$.  Finite fields
    have only the trivial valuation, so the place $Kv \to Kv_\infty$
    must be the identity map.  This forces $v_\infty = v$, in which
    case statement \ref{tres} holds.

    So, assume that $Kv$ is infinite for all definable valuations $v$.
    We will show that $Kv_\infty$ is real closed or algebraically
    closed.  Suppose not.  Then some finite extension of $Kv_\infty$
    is $p$-jammed by Remark~\ref{p-jam}, for some prime $p$.  It
    follows that $Lv_\infty(p) \ne Lv_\infty$ for all sufficiently big
    finite extensions of $K$.

    Let $v_1$ be a definable valuation on $K$ such that $Kv_1 \to
    Kv_\infty$ is pure characteristic. (We saw that such a $v_1$ must
    exist, when proving that $Kv_\infty$ is perfect.)  Let $L/K$ be a
    finite extension that is sufficiently big, so that
    \begin{itemize}
    \item $Lv_1$ has all the $p^2$th roots of unity.
    \item $Lv_\infty$ is not $p$-closed.
    \end{itemize}
    By the main theorem of $\cite{JK}$, the canonical $p$-henselian
    valuation on $Lv_1$ is definable.  By Remark~\ref{confusion}, the
    canonical $p$-henselian place is $Lv_1 \to Lv_2$ for some $v_2 >
    v_1$.  By definition of the canonical $p$-henselian place, either
    $Lv_1 \to Lv_2$ is the finest $p$-henselian place on $Lv_1$, or
    $Lv_2$ is $p$-closed.

    As $Lv_1$ has all the $p^2$th roots of unity, the same holds for
    $Lv_2$ and $Lv_\infty$.  Consequently, $p$-closedness is
    equivalent to surjectivity of the $p$th power map or surjectivity
    of the Artin-Schreier map (depending on the characteristic).  By
    choice of $v_1$, the fields $Lv_1, Lv_2, Lv_\infty$ all have the
    same characteristic.  Consequently, the place $Lv_2 \to Lv_\infty$
    ensures that
    \begin{equation*}
      Lv_2 = Lv_2(p) \implies Lv_\infty = Lv_\infty(p)
    \end{equation*}
    As $Lv_\infty$ is not $p$-closed, neither is $Lv_2$.  Therefore
    $Lv_1 \to Lv_2$ is the finest $p$-henselian valuation on $Lv_1$.

    By assumption, $Kv_2$ is infinite.  So it has dp-rank 1.  If
    $Kv_2$ is algebraically closed, then so is $Lv_2$; but we just
    showed that $Lv_2$ is not $p$-closed.  So $Kv_2$ is a dp-minimal
    field which is not algebraically closed.  By
    Theorem~\ref{summary}(2), $Kv_2$ and its finite extension $Lv_2$
    admit non-trivial henselian valuations.  So there is some
    non-trivial henselian place $Lv_2 \to k$.  The place $L \to Lv_2$
    is henselian by Theorem~\ref{summary}(3), so $Lv_1 \to Lv_2$ is
    henselian.  Compositions of henselian places are henselian, so
    $Lv_1 \to Lv_2 \to k$ is henselian, hence $p$-henselian.  Then
    $Lv_1 \to Lv_2 \to k$ is a $p$-henselian place that is strictly
    finer than $Lv_1 \to Lv_2$, contradicting the canonicity of $Lv_1
    \to Lv_2$.
  \end{enumerate}
\end{proof}

\subsection{Wrapping up} \label{sec:ksw-app}
In what follows, we will repeatedly use the \emph{Shelah expansion}.
If $M$ is an NIP structure, $M^{sh}$ denotes the expansion of $M$ by
all externally definable sets.  By \cite{NIPguide} Proposition 3.23,
$M^{sh}$ eliminates quantifiers.  Using this, it is easy to check that
$M^{sh}$ is dp-minimal when $M$ is dp-minimal.

In particular, if $K$ is our sufficiently saturated dp-minimal field,
then $K^{sh}$ is also dp-minimal (though probably no longer
saturated).

\begin{definition} \label{rough}
  A valuation $v : K \to \Gamma$ is \emph{roughly $p$-divisible} if
  $[-v(p),v(p)] \subset p \cdot \Gamma$, where $[-v(p),v(p)]$ denotes
  $\{0\}$ in pure characteristic 0, denotes $\Gamma$ in pure
  characteristic $p$, and denotes $[-v(p),v(p)]$ in mixed
  characteristic.
\end{definition}

\begin{remark} \label{2-o-3}
  Let $P$ be one of the following properties of valuation data:
  \begin{itemize}
  \item Roughly $p$-divisible
  \item Henselian
  \item Henselian and defectless
  \item Every countable chain of balls has non-empty intersection
  \end{itemize}
  If $K_1 \to K_2$ and $K_2 \to K_3$ are places, the composition $K_1
  \to K_3$ has property $P$ if and only if each of $K_1 \to K_2$ and
  $K_2 \to K_3$ has property $P$.  (In each case, this is
  straightforward to check.)
\end{remark}

\begin{definition}\label{ntdwp}
  Say that a field $K$ has \emph{nothing to do with $p$} if $p$ does
  not divide $[L : K]$ for every finite extension $L$.
\end{definition}

\begin{remark} \label{no-p} ~
  \begin{enumerate}
  \item By Corollary 4.4 of \cite{NIPfields} and
    Theorem~\ref{summary}.1, any dp-minimal field of characteristic
    $p$ has nothing to do with $p$.
  \item If $K$ has nothing to do with $p$, then any henselian
    valuation on $K$ with residue characteristic $p$ is defectless,
    and has $p$-divisible value group.
  \end{enumerate}
\end{remark}
  
\begin{lemma} \label{temp-defect}
  Let $(K,v)$ be a mixed characteristic henselian field, having
  dp-rank 1 as a valued field.  Then $v$ is defectless.  If absolute
  ramification is unbounded, then $v$ is roughly $p$-divisible, where
  $p$ is the residue characteristic.
\end{lemma}
Here, ``absolute ramification is unbounded'' means that the interval
$[-v(p),v(p)]$ is infinite in the value group.  
\begin{proof}
  Both conditions (defectlessness and rough $p$-divisibility) are
  first-order, so we may assume $(K,v)$ is saturated.  Let $\Gamma$ be
  the value group.  Let $\Delta_0$ be the biggest convex subgroup not
  containing $v(p)$, and $\Delta$ be the smallest convex subgroup
  containing $v(p)$.  (So, $\Delta_0$ is non-trivial exactly if
  absolute ramification is unbounded.)  If $K_3$ denotes the residue
  field of $K$, then these two convex subgroups decompose the place $K
  \to K_3$ as a composition of three henselian places:
  \begin{equation} \label{triple-composition}
    K \stackrel{\Gamma/\Delta}{\rightarrow} K_1
    \stackrel{\Delta/\Delta_0}{\rightarrow} K_2
    \stackrel{\Delta_0}{\rightarrow} K_3
  \end{equation}
  where each arrow is labeled by its value group.  The fields $K$ and
  $K_1$ have characteristic zero, and $K_2$ and $K_3$ have
  characteristic $p$.

  Both $\Delta_0$ and $\Delta$ are externally definable, hence
  definable in the dp-minimal field $K^{sh}$.  So the above sequence
  of places is interpretable in $K^{sh}$.  In particular, $K_2$ is a
  dp-minimal field of characteristic $p$, so $K_2$ has nothing to do
  with $p$.  By Remark~\ref{no-p}, the place $K_2 \to K_3$ is
  defectless.

  The place $K_1 \to K_2$ is defectless because it is spherically
  complete.  To see this, note that $K \to K_3$ has the countable
  chains of balls condition of Remark~\ref{2-o-3}.  Therefore so does
  $K_1 \to K_2$.  But the value group of $K_1 \to K_2$ is
  $\Delta/\Delta_0$ which is archimedean, hence has cofinality
  $\aleph_0$.  It follows that \emph{any} chain of balls has non-empty
  intersection, so $K_1 \to K_2$ is spherically complete, which
  implies henselian+defectless.

  Finally, the place $K_1 \to K_2$ is henselian defectless because is
  is equicharacteristic 0.  So, each of the three places in
  (\ref{triple-composition}) is henselian and defectless.  Therefore
  their composition $K \to K_3$ is defectless, by Remark~\ref{2-o-3}.

  Now suppose that absolute ramification is unbounded.  We first claim
  that $\Delta_0$ is $p$-divisible.  Indeed, by considering $K^{sh}$,
  one sees that $K_2$ is a NIP field, so the value group of $K_2 \to K_3$ must be $p$-divisible by Proposition 5.4 of \cite{NIPfields}.

  Let $\Delta_p$ be the largest $p$-divisible convex subgroup of
  $\Gamma$.  The group $\Delta_p$ is definable (in $K$), and it
  contains $\Delta_0$.  By unbounded ramification, $\Delta_0$ is
  \emph{not} definable (in $K$), so $\Delta_p$ is strictly bigger than
  $\Delta_0$.  Since $\Delta$ is the smallest convex group strictly
  bigger than $\Delta_0$, it follows that $\Delta_p \supseteq \Delta$
  which means $v$ is roughly $p$-divisible.
\end{proof}
This Lemma is actually true if we replace ``dp-rank 1'' with
``strongly dependent,'' which is also preserved by Shelahification,
and implies field perfection.

Now we complete the proof of Theorem~\ref{main-result}.
\begin{proof}[Proof (of Theorem~\ref{main-result})]
  Let $K$ be a sufficiently saturated dp-minimal field.

  We need to produce a valuation $v$ satisfying the conditions of
  Theorem~\ref{thm-obnoxious}:
  \begin{itemize}
  \item For every $n$, $|\Gamma/ n \Gamma|$is finite.
  \item The residue field is algebraically closed, or elementarily
    equivalent to a local field of characteristic zero.
  \item The valuation is defectless
  \item The valuation is roughtly $p$-divisible
    (Definition~\ref{rough}).
  \end{itemize}

  Any valuation on $K$ will automatically satisfy the first condition,
  by Theorem~\ref{summary}.4, so we will henceforth ignore it.

  Consider the valuation $v_\infty$ from Theorem~\ref{jk-app}.
  \begin{description}
  \item[Case 1:] The residue field $K v_\infty$ is a model of $RCF$ or
    $ACF_0$.  In this case, we take $v = v_\infty$.  Since the
    valuation is equicharacteristic 0, the defectlessness and rough
    $p$-divisibility conditions are automatic.
  \item[Case 2:] The valuation $v_\infty$ is definable and the residue
    field is finite.  By \cite{NIPfields} Proposition 5.3, $K$ must
    have characteristic zero.  Let $\Gamma$ be the value group of
    $v_\infty$, let $\Delta$ be the smallest convex subgroup
    containing $v_\infty(p)$, and let $\Delta_0$ be the largest convex
    subgroup avoiding $v_\infty(p)$.  These groups are definable in
    $K^{sh}$, so there is a $K^{sh}$-interpretable factorization of
    the place $K \to Kv_\infty$:
    \begin{equation*}
      K \stackrel{\Gamma/\Delta_0}{\rightarrow} K'
      \stackrel{\Delta_0}{\rightarrow} Kv_\infty
    \end{equation*}
    If $K'$ is infinite, the place $K' \to Kv_\infty$ would violate
    Proposition 5.3 of \cite{NIPfields}, because $K' \to Kv_\infty$ is
    interpretable in the NIP structure $K^{sh}$.  So $K'$ is finite
    and $K' \to Kv_\infty$ is trivial, making $\Delta_0$ be trivial.

    As a consequence, $\Delta = \Delta/\Delta_0$, which is
    archimedean.

    The convex subgroup $\Delta$ decomposes $K \to Kv_\infty$ as
    \begin{equation*}
      K \stackrel{\Gamma/\Delta}{\rightarrow} K''
      \stackrel{\Delta}{\rightarrow} Kv_\infty
    \end{equation*}
    By saturation, $K \to Kv_\infty$ satisfies the countable chains of
    balls condition of Remark~\ref{2-o-3}.  Therefore, so does $K''
    \to Kv_\infty$, which has archimedean value group.  Thus $K'' \to
    Kv_\infty$ is spherically complete.

    Now $K'' \to K v_\infty$ is a spherically complete field of
    characteristic zero, with finite residue field, and value group
    isomorphic to $\Zz$.  It follows that $K''$ is actually a local
    field.  We take $v$ to be the valuation corresponding to $K \to
    K''$ (i.e., $v_\infty$ coarsened by $\Delta$).  As in Case 1, the
    defectlessness and rough $p$-divisibility conditions are
    automatic.
  \item[Case 3:] The residue field of $v_\infty$ is a model of
    $ACF_p$.  In this case, we will take $v = v_\infty$.

    Let $v_1$ be some definable valuation on $K$ whose residue field
    $Kv_1$ has characteristic $p$.  (If none such exists, then $K$ has
    characteristic zero and $1/p \in \mathcal{O}$ for any definable
    valuation ring $\mathcal{O}$, so $1/p \in \mathcal{O}_\infty$
    and $v_\infty$ is equicharacteristic 0, a contradiction.)  Note
    that the valuation $v_1$ might be trivial, and might be
    $v_\infty$.

    The place $K \to Kv_\infty$ factors as $K \to Kv_1 \to Kv_\infty$
    because $v_\infty$ is finer than $v_1$.  By Remark~\ref{no-p},
    $Kv_1$ has nothing to do with $p$.  Therefore, $Kv_1 \to
    Kv_\infty$ is roughly $p$-divisible and defectless.

    By Remark~\ref{2-o-3}, it remains to see that $K \to Kv_1$ is
    roughly $p$-divisible and defectless.  By Lemma~\ref{temp-defect},
    it suffices to show that the mixed-characteristic valuation $v_1$
    has unbounded ramification.

    Suppose not.  By Theorem~\ref{summary}.4, the $p$th-power map
    $K^\times \to K^\times$ has finite cokernel. Let $\mathcal{O}$ and
    $\mathfrak{m}$ denote the valuation ring and maximal ideal of
    $v_1$. Applying the snake lemma to
    \begin{equation*}
      \xymatrix{1 \ar[r] & \mathcal{O}^\times \ar[r] \ar[d] &
        K^\times \ar[r] \ar[d] & \Gamma \ar[r] \ar[d] & 1
        \\ 1 \ar[r] & \mathcal{O}^\times \ar[r] & K^\times
        \ar[r] & \Gamma \ar[r] & 1}
    \end{equation*}
    and
    \begin{equation*}
      \xymatrix{1 \ar[r] & (1 + \mathfrak{m})^\times \ar[r] \ar[d] &
        \mathcal{O}^\times \ar[r] \ar[d] & Kv_1^\times \ar[r] \ar[d] &
        1 \\ 1 \ar[r] & (1 + \mathfrak{m})^\times \ar[r] &
        \mathcal{O}^\times \ar[r] & Kv_1^\times \ar[r] & 1}
    \end{equation*}
    where all the vertical maps are multiplication by $p$, we see that
    the $p$th power map on $(1 + \mathfrak{m})^\times$ also has finite
    cokernel.

    If there was bounded ramification, then $\mathfrak{m} = (\tau)$
    for some element $\tau$ of minimal $v_1$-valuation.  The $p$th
    power map on $1+ (\tau)$ lands in $1 + (\tau^p, p \cdot \tau)
    \subseteq 1 + \mathfrak{m}^2$.

    However, $(1 + \mathfrak{m})^\times/(1 + \mathfrak{m}^2)^\times
    \cong \mathcal{O}/\mathfrak{m} \cong Kv_1$.  So $Kv_1$ must be
    finite, which is absurd, because it has a place $Kv_1 \to
    Kv_\infty$ with algebraically closed residue field.
  \end{description}
\end{proof}

\section{VC-minimal fields} \label{sec:vc}
In \cite{dpSmall} Definition 1.4, Guingona makes the following definition:
\begin{definition}
  A theory $T$ is dp-small if there does \emph{not} exist a model $M
  \models T$, formulas $\phi_i(x;y_i)$ with $|x| = 1$, and a formula
  $\psi(x;z)$, and elements $a_{ij}, b_i, c_j$ such that
  \begin{equation*}
    M \models \phi_i(a_{i'j},b_i) \iff i = i'
  \end{equation*}
  \begin{equation*}
    M \models \psi(a_{ij'},c_j) \iff j = j'
  \end{equation*}
\end{definition}
The sort of pattern here is more general than the one in the
definition of dp-minimality, so dp-smallness is a stricter condition
than dp-minimality.

Like dp-minimality, dp-smallness is preserved under reducts and under
naming parameters.  Guingona shows that VC-minimal fields are
dp-small.

\begin{theorem}
  \label{dp-small}
  Let $K$ be a dp-small field.  Then $K$ is algebraically closed or
  real closed.
\end{theorem}
\begin{proof}
  We may (and should) take $K$ to be sufficiently saturated.  By
  Theorem 1.6.4 of \cite{dpSmall}, the value group $vK$ is divisible
  for any definable valuation $v$ on $K$.

  By Theorem~\ref{jk-app}, there is a henselian valuation $v_\infty$
  on $K$ whose valuation ring is the intersection of all definable
  valuation rings on $K$.  The residue field of $v_\infty$ is
  algebraically closed, real closed, or finite.  In the finite case,
  $v_\infty$ is definable, and we saw in the proof of
  Theorem~\ref{main-result}, that the value group of $v_\infty$ has a
  least element, so the value group $v_\infty K$ is \emph{not}
  divisible, a contradiction.

  We must therefore be in case 1 or case 3 of the proof of
  Theorem~\ref{main-result}.  In particular, $v_\infty$ is a henselian
  defectless valuation on $K$ with an algebraically closed or real
  closed residue field.  For $K$ to be algebraically closed or real
  closed, it suffices to show that $v_\infty$ has a divisible value
  group (by Ax-Kochen-Ershov in the real closed case, and
  defectlessness in the algebraically closed case).

  Let $\ell$ be any prime.  Let $a$ be an element of $K^\times$.  For
  each definable valuation $\mathcal{O}$ on $K$, the value group
  $K^\times/\mathcal{O}^\times$ is $\ell$-divisible.  So there is an
  element $b \in K^\times$ and $c \in \mathcal{O}^\times$ such that $a
  = b^\ell \cdot c$.  The valuation ring $\mathcal{O}_\infty$ of
  $v_\infty$ is the intersection of a small ordered set of
  $\mathcal{O}$'s, so by compactness, we can find $b \in K^\times$ and
  $c \in \mathcal{O}_\infty^\times$ such that $a = b^\ell \cdot c$.
  Then $v_\infty(a) = \ell \cdot v_\infty(b)$.  So $v_\infty$ has
  $\ell$-divisible value group, for arbitrary $\ell$.
\end{proof}

\section{Proof sketch of Theorem~\ref{thm-obnoxious}}
\label{sec:obnoxious}
Let $(K,v)$ be a henselian defectless field with value group $\Gamma$
and residue field $k$.  Suppose that $\Gamma/n \Gamma$ is finite for
all $n \in \Nn$.  Supppose $k$ is elementarily equivalent to a local
field of characteristic 0, or $k \models ACF_p$, in which case
$[-v(p),v(p)] \subset p \cdot \Gamma$.

Theorem~\ref{thm-obnoxious} says two things:
\begin{itemize}
\item Completeness: the theory of $(K,v)$ as a valued field is
  completely determined by the theories of $k$, $\Gamma$, and the type
  of $v(p)$ in the mixed characteristic case.
\item Dp-minimality: $(K,v)$ is dp-minimal as a valued field.
\end{itemize}

We give an outline of the proof, as a series of exercises.  In what
follows, an inclusion $(K,v) \hookrightarrow (L,v)$ of valued fields
will be called \emph{pure} if the inclusion of value groups is, i.e.,
$v(L)/v(K)$ is torsionless.  ``Definable'' will mean ``definable with
parameters'' unless specified otherwise.
\begin{enumerate}
\item Let $(\Gamma,+,<)$ be an ordered abelian group
  such that $\Gamma / n \Gamma$ is finite for all $n \in \Nn$.  Show
  that $\Gamma$ has quantifier elimination after adding unary
  predicates for \emph{all} cuts and all sets of the form $\gamma_0 +
  n \Gamma$ ($\gamma_0 \in \Gamma$).
\item Let $\Gamma$ be an ordered abelian group such that $\Gamma / n
  \Gamma$ is finite for all $n \in \Nn$.  Show that every definable
  set in $\Gamma$ is a boolean combination of sets of the form
  $\gamma_0 + n \Gamma$ and definable cuts.
\item Let $(M,v)$ be a henselian field with residue characteristic
  $p$.  Suppose $M$ has nothing to do with $p$
  (Definition~\ref{ntdwp}).  Let $K$ be a pure subfield that is
  perfect and henselian, and has separably closed residue field.
  \begin{enumerate}
  \item Show that $\Gal(K)$ is solvable.
  \item Show that $K^{alg} \cap M$ has nothing to do with $p$.
  \item Show that $K^{alg} \cap M$ is the fixed field of a Hall
    prime-to-$p$ subgroup of $\Gal(K)$.
  \end{enumerate}
\item \label{ex1} Let $T_\Gamma$ be a theory of $p$-divisible ordered
  abelian groups, having quantifier elimination in some relational
  language $L_\Gamma$.  Let $T$ be the theory of henselian defectless
  valued fields $(K,v)$ with residue field modelling $ACF_p$, and with
  value group $\Gamma$.  Show that $T$ has quantifier elimination in
  the one-sorted language of valued fields expanded by all predicates
  of the form $R(v(x_1),\ldots,v(x_n))$ where $R$ is an $n$-ary
  predicate in $L_\Gamma$.
\item Prove the completeness part of Theorem~\ref{thm-obnoxious} by
  combining \ref{ex1} with Ax-Kochen-Ershov.  (In the mixed
  characteristic case, coarsen by the largest $p$-divisible convex
  subgroup of $\Gamma$ and use both AKE and \ref{ex1}.)
\item Suppose $(M,v)$ is a $\kappa$-strongly homogeneous henselian
  valued field of residue characteristic 0.  Suppose $K$ and $L$ are
  subfields of size less than $\kappa$, and $f : K
  \stackrel{\sim}{\to} L$ is an isomorphism of valued fields, such
  that the induced maps $v(K) \to v(L)$ and $\res(K) \to \res(L)$ are
  partial elementary maps on $v(M)$ and $\res(M)$.
  \begin{enumerate}
  \item Show that the inclusion $K \hookrightarrow M$ is pure if and
    only if $L \hookrightarrow M$ is pure.
  \item If $K \hookrightarrow M$ is pure, show that $f$ extends to an
    automorphism of $M$ (i.e., $f$ is a partial elementary map).
  \end{enumerate}
\item \label{cell1} Suppose $(M,v)$ is henselian of residue
  characteristic 0.  Suppose the $n$th power map $M^\times \to
  M^\times$ has finite cokernel for all $n \in \Nn$.  Show that the
  following collection of sets generates all definable subsets of $M$
  under translations, rescalings, and boolean combinations:
  \begin{itemize}
  \item Macintyre predicates $(M^\times)^n$
  \item Sets of the form $\res^{-1}(S)$ where $S$ is a definable
    subset of $\res(M)$.
  \item Sets of the form $v^{-1}(S)$ where $S$ is a definable subset
    of $v(M)$.
  \end{itemize}
  Hint: reduce to the case where $M$ is spherically complete and
  contains representatives from all cosets of $\bigcap_{n \in \Nn}
  (M^\times)^n$.  Consider 1-types over $M$.
\item \label{cell2} Suppose that $(K,v)$ is henselian and defectless,
  $\res(K) \models ACF_p$, and the value group $v(K)$ is
  $p$-divisible.  Show that the definable subsets of $K$ are generated
  under translations, rescalings, and boolean combinations, by sets of
  the form $v^{-1}(S)$ where $S$ is a definable subset of $v(K)$.
  Hint: reduce to the case where $K$ is spherically complete, and
  consider 1-types over $K$.
\item \label{acf-cell} Let $(K,v)$ be a henselian defectless valued
  field with algebraically closed residue field and value group
  $\Gamma$.  Suppose $\Gamma / n \Gamma$ is finite for all $n$.  If
  the residue characteristic is $p > 0$, suppose furthermore that
  $[v(p),-v(p)] \subseteq p \cdot \Gamma$.  Show that every definable
  subset of $K$ is a boolean combination of sets of the form
    \begin{equation*}
      \{x : v(x-c) \in \Xi\}
    \end{equation*}
    where $\Xi$ is a definable cut in $\Gamma$, and
    \begin{equation}
      \label{acf-c}
      \{x : v(b \cdot (x-c)) \in n \cdot \Gamma\}.
    \end{equation}
    Hint: combine \ref{cell1} and \ref{cell2}.
\item \label{rcf-cell} Let $(K,v)$ be a henselian valued field with
  real-closed residue field, and value group $\Gamma$ satisfying
  $\Gamma / n \Gamma$ finite for all $n$.
  \begin{enumerate}
  \item Show that $K$ admits finitely many orderings, which are all
    definable.
  \item Fix some ordering.  Show that every definable subset of $K$ is
    a boolean combination of sets of the form $\{x : x > c\}$,
    \begin{equation*}
      \{x : v(x-c) \in \Xi\}
    \end{equation*}
    where $\Xi$ is a definable cut in $\Gamma$, and
    \begin{equation}
      \label{rcf-c}
      \{x : v(b \cdot (x-c)) \in n \cdot \Gamma\}
    \end{equation}
    Hint: combine \ref{cell1} with o-minimality of RCF.
  \end{enumerate}
\item \label{pcf-cell} Let $(K,v)$ be a henselian valued field of
  mixed characteristic, with finite residue field and bounded absolute
  ramification.  Let $\Gamma$ be the value group, and suppose $\Gamma
  / n \Gamma$ is finite for all $n$.  Show that every definable subset
  of $K$ is a boolean combination of sets of the form
  \begin{equation}
    \{x : v(x - c) \in \Xi\}
  \end{equation}
  where $\Xi$ is a definable cut in $\Gamma$, and
  \begin{equation}
    \label{pcf-c}
    \{x : P_n(b \cdot (x-c)) \}
  \end{equation}
  where $P_n$ is the $n$th Macintyre predicate.
\item \label{acvf-cut} In the setting of \ref{acf-cell} or
  \ref{pcf-cell}, let $L_0$ be the language of valued fields expanded
  with all predicates of the form $v^{-1}(\Xi)$ for $\Xi$ a definable
  cut in the value group.  Show that $(K,v)$ is dp-minimal on the
  level of quantifier-free $L_0$ formulas.  More specifically, show
  that if $b_1, b_2, \ldots$ and $c_1, c_2, \ldots$ are mutually
  indiscernible sequences of tuples from $K$ and $a \in K^1$, then
  either $b_1, b_2, \ldots$ or $c_1, c_2, \ldots$ is
  quantifier-free-$L_0$-indiscernible over $a$.  Hint: ACVF remains
  dp-minimal after naming all cuts in the value group.
\item \label{rcvf-cut} In the setting of \ref{rcf-cell}, let $L_0$ be
  the language of valued fields expanded with a binary predicate for
  one of the orderings, as well as all predicates of the form
  $v^{-1}(\Xi)$ for $\Xi$ a definable cut in the value group.  Show
  that $(K,v)$ is dp-minimal on the level of quantifier-free $L_0$
  formulas.  More specifically, show that if $b_1, b_2, \ldots$ and
  $c_1, c_2, \ldots$ are mutually indiscernible sequences of tuples
  from $K$ and $a \in K^1$, then either $b_1, b_2, \ldots$ or $c_1,
  c_2, \ldots$ is quantifier-free-$L_0$-indiscernible over $a$.  Hint:
  RCVF remains dp-minimal after naming all cuts in the value group.
\item \label{step-1} In the setting of \ref{acf-cell}, \ref{rcf-cell},
  or \ref{pcf-cell}, suppose $(K,v)$ fails to be dp-minimal, and is
  sufficiently saturated.  Show that there is an indiscernible
  sequence $\langle C_i \rangle_{i \in \Zz}$ of sets of the form
  (\ref{acf-c}), (\ref{rcf-c}), or (\ref{pcf-c}), respectively, and an
  element $a \in K$ which belongs to some but not all of the $C_i$'s,
  such that the sequence $\langle c_i \rangle_{i \in \Zz}$ of
  ``centers'' of the $C_i$'s is $L_0$-quantifier-free indiscernible
  over $a$.  (Hint: take a failure of dp-minimality.  Write the sets
  explicitly in terms of the ``cell-like'' sets given in
  \ref{acf-cell}, \ref{rcf-cell}, or \ref{pcf-cell}, respectively.
  Arrange for the parameters to be mutually indiscernible and
  $\Zz$-indexed.  Let $a$ be an element which is in the zeroth set in
  the first row and the zeroth set in the second row, but in no
  others.  Use \ref{acvf-cut} or \ref{rcvf-cut} to choose one of the
  two rows of sets whose parameter sequence is still
  $L_0$-indiscernible over $a$.  Using the explicit expression of the
  sets in the chosen row in terms of the ``cell-like'' sets, find
  ``cell-like'' $C_i$'s.  By choice of $L_0$, they must be of the form
  (\ref{acf-c}), (\ref{rcf-c}), or (\ref{pcf-c}), rather than
  $v^{-1}(\Xi)$ for $\Xi$ a definable cut.)
\item \label{step2-ar} In the setting of \ref{step-1} plus
  \ref{acf-cell} or \ref{step-1} plus \ref{rcf-cell}, argue that $v(a
  - c_i)$ depends on $i$.  Using the $L_0$-indiscernibility of the
  $i$'s over $a$, conclude that $a$ is a pseudo-limit of the $c_i$'s
  as $i$ goes to $+\infty$, or as $i$ goes to $-\infty$.  In the
  former case, argue that $v(a - c_i) = v(c_N - c_i)$ for $N > i$.
  Use full indisernibility of the $c_i$'s over $\emptyset$ to conclude
  that the coset
  \begin{equation*}
    v(c_i - c_j) + \bigcap_{n \in \Nn} n \Gamma
  \end{equation*}
  doesn't depend on $i \ne j$.  Conclude that $v(a - c_i) - v(a - c_j)
  \in \bigcap_{n \in \Nn} n \Gamma$ for all $n$ and obtain a
  contradiction.
\item \label{step2-p} In the setting of \ref{step-1} plus
  \ref{pcf-cell}, argue that $rv_n(a - c_i)$ depends on $i$.  Using
  the $L_0$-indiscernibility of the $i$'s over $a$, argue that $rv_n(a
  - c_i) \ne rv_n(a - c_j)$ for $i \ne j$.  Use the finiteness of the
  projection $rv_n(K) \to v(K)$ to argue that $v(a - c_i) \ne v(a -
  c_j)$ for $i \ne j$.  As in \ref{step2-ar} argue that the $c_i$
  pseudoconverge to $a$, perhaps after reversing their order.
  Conclude that $rv_n(a - c_i) = rv_n(c_N - c_i)$ for $N \gg i$.  Use
  full indiscernibility of the $c_i$'s over $\emptyset$ to conclude
  that the coset
  \begin{equation*}
    rv_n(c_i - c_j) + \bigcap_{m \in \Nn} m \cdot rv_n(K)
  \end{equation*}
  doesn't depend on $i, j$.  Conclude that $rv_n(a - c_i) - rv_n(a -
  c_J) \in \bigcap_{m \in \Nn} m \cdot rv_n(K)$ and obtain a
  contradiction.
\item If $(K,v)$ is henselian with suitable value group and a residue
  field that is non-archimedean local of characteristic 0, prove that
  $(K,v)$ is dp-minimal by viewing $v$ as a coarsening of a valuation
  $w$ such that $(K,w)$ is as in \ref{pcf-cell}.
\end{enumerate}
We also remark that the equicharacteristic zero case of
Theorem~\ref{thm-obnoxious} has been proven in unpublished work by
Chernikov and Simon: they show that an equicharacteristic 0 valued
field with inp-minimal value group and residue field is itself
inp-minimal.


\renewcommand{\abstractname}{Acknowledgments}
\begin{abstract}
  I'd like to thank Tom Scanlon for telling me about \cite{JK} and
  \cite{NIPfields}, and Nick Ramsey for informing me of the work of
  Jahnke, Simon, and Walsberg (soon to appear as \cite{JSW}).
\end{abstract}

\bibliographystyle{plain}
\bibliography{mybib}{}

\end{document}